\documentclass[table]{amsart}
\usepackage[margin=1.1in]{geometry}
\usepackage{amscd,amsmath,amsxtra,amsthm,amssymb,stmaryrd,xr,mathrsfs,mathtools,enumerate,commath, comment, mathtools}
\usepackage{tikz}
\usetikzlibrary{calc}
\usepackage{tikz-3dplot}
\usepackage{stmaryrd}
\usepackage{multirow}
\usepackage{xcolor}
\usepackage{commath}
\usepackage{comment}
\usepackage{svg}
\usepackage{graphics}
\usepackage{longtable} 
\usepackage{pdflscape} 
\usepackage{booktabs}
\usepackage{hyperref}
\definecolor{vegasgold}{rgb}{0.77, 0.7, 0.35}
\definecolor{darkgoldenrod}{rgb}{0.72, 0.53, 0.04}
\definecolor{gold(metallic)}{rgb}{0.83, 0.69, 0.22}
\hypersetup{
 colorlinks=true,
 linkcolor=darkgoldenrod,
 filecolor=brown,      
 urlcolor=gold(metallic),
 citecolor=darkgoldenrod,
 }
\newtheorem{lthm}{Theorem}

\usepackage[all,cmtip]{xy}

\DeclareFontFamily{U}{wncy}{}
\DeclareFontShape{U}{wncy}{m}{n}{<->wncyr10}{}
\DeclareSymbolFont{mcy}{U}{wncy}{m}{n}
\DeclareMathSymbol{\Sh}{\mathord}{mcy}{"58}
\usepackage[T2A,T1]{fontenc}
\usepackage[OT2,T1]{fontenc}

\newtheorem{theorem}{Theorem}[section]
\newtheorem{lemma}[theorem]{Lemma}

\newtheorem*{theorem*}{Theorem}
\newtheorem*{ass*}{Assumption}
\newtheorem{definition}[theorem]{Definition}
\newtheorem{corollary}[theorem]{Corollary}
\newtheorem{remark}[theorem]{Remark}
\newtheorem{example}[theorem]{Example}

\newtheorem{proposition}[theorem]{Proposition}

\newcommand{\cK}{\mathcal{K}}

\newcommand{\Z}{\mathbb{Z}}
\newcommand{\Q}{\mathbb{Q}}
\newcommand{\F}{\mathbb{F}}

\newcommand{\cL}{\mathcal{L}}

\newcommand{\cO}{\mathcal{O}}

\newcommand{\Spec}{\mathrm{Spec}\ }

\newcommand{\cover}{\Sigma_2}

\newcommand{\op}[1]{\operatorname{#1}}

 \DeclareMathSymbol{\sha}{\mathord}{mcy}{"58}
 \makeatletter
\newcommand{\mylabel}[2]{#2\def\@currentlabel{#2}\label{#1}}
\makeatother

\numberwithin{equation}{section}

\begin{document}

\author[A.~Ray]{Anwesh Ray}
\address[Ray]{Chennai Mathematical Institute, H1, SIPCOT IT Park, Kelambakkam, Siruseri, Tamil Nadu 603103, India}
\email{anwesh@cmi.ac.in}

\title[]{Arithmetic invariants of torus links}
\author[T.~Shah]{Tanushree Shah}
\address[Shah]{Chennai Mathematical Institute, H1, SIPCOT IT Park, Kelambakkam, Siruseri, Tamil Nadu 603103, India}
\email{tanushree@cmi.ac.in}


\maketitle

\begin{abstract}
The classical analogy between knots and primes motivates the study of Alexander polynomials through an arithmetic perspective. In this article we study the two–parameter family of torus knots and links \(T_{p,q}\) and analyze the asymptotic behaviour of the zeros of their Alexander polynomials \(\Delta_{p,q}(t)\), defined with respect to the total linking number covering.
 We prove that as $p,q\rightarrow \infty$, these zeros become equidistributed on the unit circle and derive an explicit formula for the limiting frequency with which primitive \(r\)-th roots of unity appear. To capture finer statistical information, we introduce the moment sequence of the zero distribution and compute its generating function in closed form. We further examine the Iwasawa theory of the corresponding branched covers, determining the Iwasawa invariants. The logarithmic Mahler measure of \(\Delta_{p,q}(t)\) vanishes identically and the associated homological growth in towers of abelian covers of \(S^3\) branched along \(T_{p,q}\) is subexponential.
\end{abstract}

\section{Introduction}
\subsection{Motivation and background}
\par The study of knots and links in $S^3$ exhibits analogies with the arithmetic of prime ideals in number rings, a perspective developed under the name \emph{arithmetic topology}. The analogy was first observed by Mazur in his unpublished notes in 1963 \cite{mazur1963remarks}, which documents conversations with Mumford. This analogy was further developed by Kapranov \cite{kapranov1996analogies} and Resnikov in \cite{Reznikov1,Reznikov2}. Under this correspondence, a knot (or link) in $S^{3}$ is viewed as the topological analogue of a prime ideal, and the structure of covering spaces of $S^{3}$ branched along that knot reflects the behavior of number field extensions with prescribed ramification. In particular, the Galois theory of such branched covers serves as a topological counterpart to the Galois theory of number fields with controlled inertia and decomposition at a fixed prime (cf.\ \cite{morishita2011knots}). This has led to the recognition that many classical invariants of knot and link complements admit arithmetic analogues, often arising from Galois-theoretic and cohomological constructions. For instance, the study of the Alexander module of a knot has parallels with Iwasawa modules over the cyclotomic $\mathbb{Z}_p$-extension of a number field.

\par Many link invariants can be studied in terms of the roots of the Alexander polynomials. For instance, the location, multiplicity, and distribution of these roots reflect subtle features of the link, such as \emph{fiberedness}, symmetry, and the growth asymptotics of its infinite cyclic and abelian covers. More precisely, the Mahler measure of a multivariable Alexander polynomial of a link $\cL$ is related to asymptotic growth in homology of abelian covers of $S^3$ that are branched along $\mathcal{L}$ (cf. \cite{SW1, SW2}). It so happens that the Alexander polynomial itself is entirely analogous to the Iwasawa polynomial or the $p$-adic $L$-function on the arithmetic side. Thus the Alexander polynomial should not only be viewed as a topological invariant, but also as an object of genuine arithmetic complexity. Its study naturally leads to questions about equidistribution, growth of torsion, and spectral statistics that are as arithmetically flavored as they are topological. This leads us to tread on new ground, as systematic investigations of probabilistic phenomena in low–dimensional topology are far less developed than their arithmetic counterparts. 

\par The central objects of study in this paper are the Alexander polynomials of torus links, which have proven to be a testing ground for various conjectures in low-dimensional topology. This is because many of their invariants admit explicit algebraic and geometric descriptions. By varying the pair of coprime integers \((p,q)\), we obtain a two--parameter family of links whose algebraic and topological complexity increases in a controlled manner. Moreover, for any bounded region in \(\mathbb{R}^2\), there are only finitely many such pairs \((p,q)\), and hence only finitely many torus links in the family with parameters in that region. This finiteness property makes the family amenable to statistical investigation: it is the topological analogue of the \emph{Northcott property} in arithmetic geometry, which asserts that the set of points of bounded height on a projective variety defined over a number field is finite. A second feature is that Alexander polynomials of torus links factor completely into cyclotomic polynomials, and therefore their zeroes consist of Galois orbits of roots of unity, counted with multiplicity. In this sense, although Alexander polynomials arise from the topology of \(S^3\), their behaviour is reflected in the Galois theory of cyclotomic fields. 
\subsection{Main results}
\par Next, let us describe some of the main results of this article. 
\par We begin by showing that the zeros of the Alexander polynomials of torus knots become equi-distributed on the unit circle as the knot parameters grow. Writing a torus knot as $T_{p,q}$ with coprime positive integers $(p,q)$, the height function $\operatorname{ht}(p,q)=\max\{p,q\}$ is used to define the finite family $\mathcal{T}_1(X)$ of knots of height at most $X$, whose cardinality is asymptotically \[\#\mathcal{T}_1(X)\sim X^2/\zeta(2).\] Since the Alexander polynomial of $T_{p,q}$ has $(p-1)(q-1)$ roots on the unit circle, one is naturally led to study the set $\Omega_1(X)$ of all roots arising from knots in $\mathcal{T}_1(X)$, and it is shown that \[\#\Omega_1(X)\sim X^4/(4\zeta(2)).\] Denote by $\Delta_{p,q}(t)$ the Alexander polynomial of the knot $T_{p,q}$.

\begin{lthm}[Theorem \ref{section 3 main thm 1}]
Let $\mathbb{T}=\{z\in\mathbb{C}\mid |z|=1\}$ and write each $z\in\mathbb{T}$ as $z=e^{2\pi i\theta(z)}$ with $\theta(z)\in[0,1)$. For any interval $[a,b]\subset[0,1]$,
\[
\#\{(T_{p,q},\alpha)\in\Omega_1(X)\mid \theta(\alpha)\in[a,b]\}
= \frac{(b-a)}{4\zeta(2)}X^4 + O(X^3\log X).
\]
\end{lthm}
\noindent We also prove a similar result for torus links, stated below. Let $(p,q)$ be a pair of positive integers and let $d:=(p,q)$. Then the torus link $T_{p,q}$ has $d$ component links and the Alexander polynomial of $T_{p,q}$ is a multivariable polynomial in $d$-variables 
\[\Delta_{T_{p,q}}(X_1,\dots, X_d)\in \Z[X_1^{\pm 1}, \dots, X_d^{\pm 1}].\] When $d\geq 2$, we multiplying by $(t-1)$ and set all the variables $X_i:=t$ we obtain a single variable Alexander polynomial
\[\Delta_{p,q}(t):=(t-1)\Delta_{T_{p,q}}(t, \dots, t)\in \Z[t^{\pm 1}].\]
\begin{lthm}[Theorem \ref{section 3 main thm 2}]
Let $\mathcal{T}(X)$ be the set of all torus links $T_{p,q}$ with $1\le p,q\le X$, and let $\Omega(X)$ (resp.\ $\Omega_{[a,b]}(X)$) be the multiset of pairs $(T_{p,q},\alpha)$ where $\alpha$ is a root of $\Delta_{p,q}(t)$ (resp.\ a root lying on the arc $\{e^{2\pi i\theta}:\theta\in[a,b]\}$). Then for every closed interval $[a,b]\subset[0,1]$ one has
\[
\#\Omega_{[a,b]}(X)=\frac{(b-a)}{4}X^{4}+O(X^{3}),
\]
and in particular
\[
\lim_{X\to\infty}\frac{\#\Omega_{[a,b]}(X)}{\#\Omega(X)} = b-a.
\]
\end{lthm}

This leads naturally to a measure-theoretic interpretation: the finite measures supported on the roots of 
$\Delta_{p,q}(t)$ converge weakly to the normalized Haar measure $\mu_{\mathbb{T}}$ on the unit circle $\mathbb{T}$. 

\begin{lthm}[Theorem \ref{measure theoretic convergence theorem}]
Let $\mu_X$ and $\mu_X'$ be the probability measures on $\mathbb{T}$ defined by
\[
\mu_X \;=\; \frac{1}{\#\Omega(X)} \sum_{(T_{p,q},\alpha)\in\Omega(X)} \delta_{\alpha},
\qquad
\mu_X' \;=\; \frac{1}{\#\Omega_1(X)} \sum_{(T_{p,q},\alpha)\in\Omega_1(X)} \delta_{\alpha},
\]
where $\delta_{\alpha}$ denotes the Dirac mass at $\alpha\in\mathbb{T}$. Then both $\mu_X$ and $\mu_X'$ converge \emph{weakly} to $\mu_{\mathbb{T}}$ as $X\to\infty$.
\end{lthm}
A finer form of equidistribution is obtained by isolating the contribution of roots of a fixed order.  
Since every zero of $\Delta_{p,q}(t)$ is a root of unity, one may ask not only how the full set of roots distributes on $\mathbb{T}$, but how often a \emph{given} cyclotomic order occurs as $(p,q)$ varies. This leads to the following explicit formula, which shows that the distribution of cyclotomic orders is governed purely by $\omega(r)$, the number of distinct prime factors for $r$.

\begin{lthm}[Theorem \ref{thm:avg-Mr}]
Fix an integer $r\ge2$, and let $\omega(r)$ denote the number of distinct prime divisors of $r$. Then the limiting frequency
\[
\mathcal{F}_r(X)=\frac{1}{\#\mathcal{T}_1(X)}\sum_{\substack{1\le p,q\le X\\ \gcd(p,q)=1}} \mathbf{1}_r(p,q)
\]
exists, and one has
\[
\lim_{X\to\infty}\mathcal{F}_r(X)=\frac{2^{\omega(r)}-2}{r}.
\]
Equivalently, the average multiplicity of primitive $r$-th roots of unity among the zeros of $\Delta_{p,q}(t)$ tends to
\[
\frac{\varphi(r)\bigl(2^{\omega(r)}-2\bigr)}{r}.
\]
\end{lthm}
To study the distribution of the zeros of $\Delta_{p,q}(t)$ quantitatively, it is natural to consider their power sums, or \emph{moments}, defined by
\[
S_m(p,q) \;=\; \sum_{\alpha\ \text{zero of}\ \Delta_{p,q}} \alpha^m,
\qquad m\ge0,
\]
where each zero is counted with multiplicity. We package the sequence of moments into the generating function
\[
G_{p,q}(z)\;=\;\sum_{m\ge0} S_m(p,q)\,z^m,
\]
whose analytic properties reflect the arithmetic structure of the roots.  

\begin{lthm}[Theorem \ref{Theorem G p q residues}]
For $|z|<1$, the generating function
\[
G_{p,q}(z)=\sum_{m\ge0} S_m(p,q)\,z^m
\]
admits the closed form
\[
G_{p,q}(z)=\frac{pq}{1-z^{pq}}-\frac{p}{1-z^p}-\frac{q}{1-z^q}+\frac{1}{1-z},
\]
and extends meromorphically to $\mathbb{C}$ with simple poles at the roots of unity of orders $1,p,q,pq$. For a root of unity $\xi$,
\[
\operatorname{Res}_{z=\xi}G_{p,q}(z)
= -\xi^{-(pq-1)}+\xi^{-(p-1)}+\xi^{-(q-1)}-1.
\]
\end{lthm}

The moments $S_m(p,q)$ repeat with exact period $pq$ and over one full period their average value is zero while their mean square is $(p-1)(q-1)$. In fact, by expressing the discrete Fourier transform of the moment sequence in terms of the residues of the generating function $G_{p,q}(z)$ at the $pq$-th roots of unity, one finds that
\[
\frac{1}{pq}\sum_{m=0}^{pq-1}\lvert S_m(p,q)\rvert^2
=
\sum_{\omega^{pq}=1}\lvert \operatorname{Res}_{z=\omega}G_{p,q}(\omega)\rvert^2,
\]
\noindent so that the size of the fluctuations of the moments is governed entirely by the pole data of $G_{p,q}(z)$. As $p, q\rightarrow \infty$, the variance approaches $\infty$ at a linear rate, as does the number of poles on the unit circle.

\par In the setting of cyclic and abelian covers of link complements, the torsion growth in first homology is controlled by the Alexander polynomial. Silver and Williams \cite{SW1, SW2} related asymptotics of torsion to the logarithmic Mahler measure of the multivariable Alexander polynomial. For torus links, we obtain the following result.

\begin{lthm}[Corollary \ref{cor:torus-measure-zero} and Corollary \ref{cor 4.8}]
Let $T_{p,q}$ be a torus link with $d=\gcd(p,q)$ components, and for each $n\ge1$ let $M_n$ denote the Fox–completed $(\Z/n\Z)^d$-cover of $S^3$ branched along $T_{p,q}$. Then the multivariable Alexander polynomial $\Delta_{p,q}$ satisfies
\[
m\big(\Delta_{p,q}\big)=0,
\]
and hence
\[
\lim_{n\to\infty}\frac{1}{n^d}\log\#\bigl(H_1(M_n;\Z)_{\mathrm{tors}}\bigr)=0.
\]
\end{lthm}
The Iwasawa theory of links arises from the observation that the first homology groups of cyclic or abelian covers of a link complement behave, in many respects, like the ideal class groups in towers of number fields. For a link $\cL\subset S^3$ with $r$ components, the abelianization of its fundamental group gives a canonical map $\pi_1(S^3\setminus \cL)\twoheadrightarrow \Z^r$, and taking successive quotients by $\ell^n\Z^r$ produces a natural $\Z_\ell^r$–tower of covers. The growth of the $\ell$–primary torsion in $H_1$ along this tower is encoded in a finitely generated torsion module over the multi–variable Iwasawa algebra $\Z_\ell[[T_1,\dots,T_r]]$, whose characteristic power series is obtained from the multivariable Alexander polynomial of $\cL$ by $\ell$–adic specialization. In exact parallel with the number–field case, the structure theory of such modules gives rise to Iwasawa invariants $\mu, \lambda, \nu$ that measure the $\ell$–adic growth in homology. Thus the Alexander polynomial plays the role analogous to that of the $p$–adic $L$–function. In the context of torus links, we have the following explicit result.
\begin{lthm}[Theorems \ref{thm 5.3} and \ref{thm 5.4}]
Let $T_{p,q}$ be the $(p,q)$–torus link, and let $d=\gcd(p,q)$. An integral vector $z=(z_1,\dots, z_d)\in \Z^r$ is said to be \emph{admissible} if $\op{gcd}(z_1,\dots, z_d)=1$ and $\prod_i z_i\neq 0$. Given an admissible vector $z$, let $\mu_z(p,q)$, $\lambda_z(p,q)$, and $\nu_z(p,q)$ denote the corresponding Iwasawa invariants. When $d=1$, these invariants are independent of the choice of $z\in\{\pm 1\}$ and we simply denote them by $\mu(p,q)$, $\lambda(p,q)$, and $\nu(p,q)$ respectively. 
\begin{description} \item[Knot case]If $\gcd(p,q)=1$, so that $T_{p,q}$ is a knot, then
\[
\mu(p,q)=\lambda(p,q)=\nu(p,q)=0
\]
for every $\ell$–power root of unity $z$. In particular, all Iwasawa invariants vanish identically in the knot case.
\item[Link case] If $d=\gcd(p,q)\ge2$, then $\mu_z(p,q)=0$ for every admissible vector $z$. Set $\alpha:=\sum_{i=1}^d z_i$, then the $\lambda$–invariant is given by
\[
\lambda_z(p,q)=
\begin{cases}
(d-2)\,\ell^{\,v_\ell(\alpha)}  & \text{if } z^\alpha=1 \text{ for some } \alpha\mid pq,\\[4pt]
0 & \text{otherwise}.
\end{cases}
\]
\end{description}
\end{lthm}

\subsection{Outlook}
\par While torus links provide a uniquely explicit testing ground, it would be of considerable interest to extend these statistical phenomena to broader families of links, seeking equidistribution results for the roots of their Alexander polynomials. A second direction concerns studying finer invariants like twisted Alexander polynomials, attached to a representation of the link group and the relationship that this may have with various number fields of interest. The fact that the Alexander polynomials of torus knots have roots consisting of Galois orbits in cyclotomic fields leaves much to ponder about. The results proved in this article should therefore be viewed as an invitation to pursue a richer probabilistic framework for more general links and $3$–manifolds.

\section{Preliminary notions}

\par In this section, we establish the notation and recall the preliminary background that will be used in the remainder of the paper.
\subsection{Torus knots}
\par To begin, note that any two knot diagrams of the same link, up to planar isotopy, can be related by a finite sequence of \emph{Reidemeister moves} (cf. \cite[Fig. 1.4]{Lickorish}). 
A torus knot is a knot that can be drawn on the surface of a standard torus in three-dimensional space. Given two coprime non-zero integers $p,q$, the $(p,q)$-torus knot $T_{p,q}$ is defined as the set of points on the torus obtained by winding $p$ times around the meridional direction and $q$ times around the longitudinal direction before closing up. More generally, let $d:=\op{gcd}(p, q)$ denote the greatest common divisor of $p$ and $q$. If $d>1$, the corresponding curve does not form a knot but rather a link with $d$ components, which we shall also denote by $T_{p,q}$. The simplest examples include the trefoil knot, which is the $(2,3)$-torus knot, and the cinquefoil knot, the $(2,5)$-torus knot shown in Figures \ref{2,3} and \ref{2,5} respectively. 
Torus knots are fibered knots, which means that their complements in $S^3$ have the structure of a surface bundle over the circle. The minimal Seifert surface associated to the knot has genus $\frac{(p-1)(q-1)}{2}$, and its boundary is the knot itself. The fundamental group of the complement of $T_{p,q}$ in $S^3$ is given by:
$$
\pi_1(S^3 \setminus T_{p,q}) = \langle x, y \mid x^p = y^q \rangle.
$$
\noindent Here, $x$ represents a meridional loop around the torus, and $y$ represents a longitudinal loop along the torus. The single relation $x^p = y^q$ encodes the fact that the knot winds $p$ times meridionally and $q$ times longitudinally around the torus. This simple presentation already illustrates the special algebraic structure of torus knots, making them convenient examples for studying invariants of knots.

\begin{figure}
   \centering
    \includegraphics[width=1in]{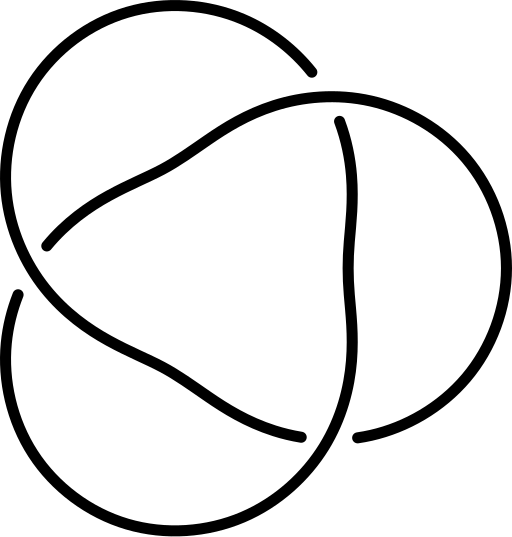}
     \caption{Trefoil: torus (2,3) knot}
    \label{2,3}
\end{figure}

\begin{figure}
   \centering
    \includegraphics[width=1in]{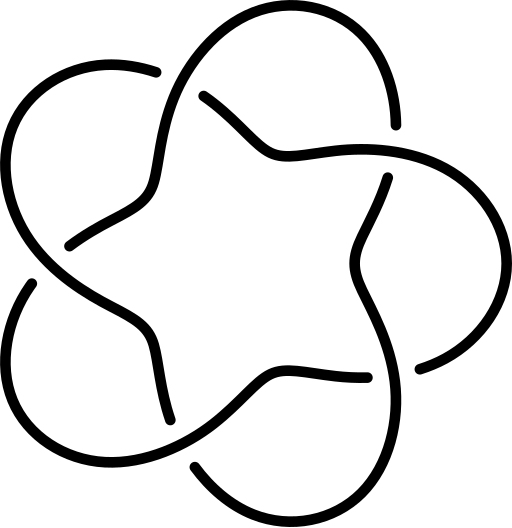}
     \caption{Cinquefoil: torus (2,5) knot}
    \label{2,5}
\end{figure}


\subsection{Alexander polynomials and coloring invariants}
\par Let $\cL$ be a link, i.e., an embedding of a finite disjoint union of circles into the $3$-sphere. Let $D$ be a knot diagram for $\cL$. Given a prime number $\ell$, we recall the notion of an $\ell$-coloring of $D$.
\begin{definition}
An \emph{$\ell$-coloring} of $D$ is an assignment of a residue class $m\in \Z/\ell\Z$ to each strand of $D$ such that whenever we have a crossing with the overstrand associated to $y$, 
and the understrands associated to $x$ or $z$ we have the relation $x+z = 2y$ in $\Z/\ell\Z$.
\end{definition}
\noindent Let $C_1, \dots, C_n$ be an enumeration of the strands in the diagram $D$ and let $x_i$ be the color associated to $C_i$. Then they have the equation $x_i+x_k=2x_j$ at every crossing. This gives a system of equations $C \vec{x}=0$, where $C$ is an $n\times n$ matrix. An $\ell$-coloring is trivial if all strands are colored the same. Two colorings $f_1, f_2: D\rightarrow \Z/\ell\Z$ are equivalent if there is a fixed number $t\in \Z/\ell \Z$ such that $f_2(C_i)=f_1(C_i)+t$. Denote by $C_\ell(D)$ the set of equivalence classes of non-trivial $\ell$-colorings of $D$. Let $n(D)$ be the cardinality of $C(D)$. 

\par Since any two diagrams of the same link differ by a sequence of Reidemeister moves, it follows that the study of link invariants reduces to studying invariants of link diagrams that remain unchanged under these moves. In particular, colorings of link diagrams extend uniquely across each Reidemeister move. In greater detail, given a coloring before applying a move, there is a unique compatible coloring after the move, and vice versa. Hence, the notion of an $\ell$-coloring is well defined for links and is independent of the chosen diagram. 
\par Let $C$ be the $n\times n$ matrix defined above. Note that $C$ has the property that the sum of all rows and the sum of all columns are zero vectors. Any $(n-1)\times (n-1)$ submatrix $C'$ of $C$ obtained by deleting a column and a row has the same determinant. The determinant of $C'$ is a link invariant and is denoted by $\op{det}(\cL)$. It is easy to see that $\ell$ divides $\op{det}(\cL)$ if and only if there is a non-trivial $\ell$-coloring of $\cL$. 

\par The definition of $\ell$–colorings is purely combinatorial, but it has a close connection with the notion of the \emph{Alexander polynomial}, which is one of the most classical link invariants. Recall that the link group is the fundamental group $\op{G}_{\cL}:=\pi_1(S^3\setminus \mathcal L)$. The abelianization $\op{G}_{\cL}^{\op{ab}}$ is isomorphic to $\mathbb Z^r$, where $r$ is the number of components of $\mathcal L$, and the abelianization map sends each meridian to a standard basis vector.
\par Given a diagram $D$, the Wirtinger presentation provides a concrete way to describe the fundamental group of the link complement. Each strand of $D$ is assigned a generator represented by homotopy class of a small meridian loop encircling that strand. At each crossing, one obtains a relation expressing the generator corresponding to the overstrand as a conjugate of one understrand by the other. More explicitly, if the overstrand carries generator $x_j$ and the understrands carry generators $x_i$ and $x_k$, then the crossing contributes the relation $x_jx_ix_j^{-1}=x_k$. Applying Fox calculus to the Wirtinger presentation produces a matrix $F=\left(\partial r_j/\partial x_i\right)_{i,j}$ whose entries lie in the integral group ring $\mathbb{Z}[F_n]$, where $F_n$ is the free group on the generators $x_1,\dots,x_n$ associated to the strands of the diagram.

\par Suppose \(M\) is a finitely presented module over a commutative Noetherian ring \(R\), with presentation
\[
R^m \xrightarrow{A} R^n \longrightarrow M \longrightarrow 0,
\]
where \(A\) is an \(n\times m\) matrix with entries in \(R\). The \(k\)-th elementary ideal of \(M\), denoted \(E_k(M)\), is the ideal in \(R\) generated by all minors of size \((m-k)\times(m-k)\) of the matrix \(A\). These ideals are independent of the chosen presentation, up to multiplication by units. The \emph{order ideal} of \(M\) is the smallest nonzero elementary ideal, that is, \(E_0(M)\), which is generated by the maximal minors of \(A\). When \(M\) has rank one over \(R\), this ideal is principal, and a generator of this ideal is called the \emph{order ideal} of \(M\).

\par For a link with $r$ components, one abelianizes $G$ not to $\mathbb Z$, as in the knot case, but to $\mathbb Z^r$, sending each meridian generator $x_i$ to a variable $t_{c(i)}$, where $c(i)$ denotes the link component of the strand corresponding to $x_i$. This yields a ring homomorphism $
\mathbb{Z}[F_n] \longrightarrow \mathbb{Z}[t_1^{\pm1},\dots,t_r^{\pm1}],
$ under which the entries of the Fox Jacobian become Laurent polynomials in $r$ commuting variables. The resulting matrix is the \emph{Alexander matrix} of the link. Its cokernel defines the \emph{Alexander module}, which is a finitely generated module over $\mathbb{Z}[t_1^{\pm1},\dots,t_r^{\pm1}]$. The structure of this module encodes subtle information about the topology of the link complement, and its order ideal yields the \emph{multivariable Alexander polynomial} $\Delta_{\mathcal L}(t_1,\dots,t_r)$, which is well defined up to multiplication by monomials $\pm t_1^{a_1}\cdots t_r^{a_r}$. 

\par The relation to colorings appears when one specializes all variables $t_i$ to a single parameter $t$, corresponding to the diagonal abelianization $G\to \mathbb Z$ sending every meridian to $1$. In this case the Alexander matrix has entries in $\mathbb Z[t,t^{-1}]$. When $r\geq 2$, set $\Delta_{\cL}(t):=(t-1)\Delta_{\mathcal L}(t,\dots,t)$. At a crossing with overstrand generator $x_j$ and understrand generators $x_i$ and $x_k$, the associated Wirtinger relation is $x_jx_ix_j^{-1}x_k^{-1}$. The number of strands equals the number of crossings, however, one can always ignore one of the crossings since the associated relation is captured by the other crossings. Thus, $m=n-1$, and computing Fox derivatives gives coefficients $t$, $1-t$, and $-1$ in the columns corresponding to $i,j,k$. When one substitutes $t=-1$, these coefficients become $-1$, $2$, and $-1$, which are precisely the coefficients of the coloring relation $x_i+x_k-2x_j=0$. Thus the coloring matrix arises from the Alexander matrix by specializing at $t=-1$, up to elementary row and column operations. In particular, the reduced coloring matrix $C'$ coincides up to sign with the Alexander matrix $A(t)$ evaluated at $t=-1$. It follows that 
\begin{equation}
    \det(\mathcal L) = \pm \Delta_{\mathcal L}(-1)
\end{equation} and thus, $\cL$ admits a nontrivial $\ell$–coloring if and only if $\ell$ divides $\Delta_{\mathcal L}(-1)$. In general, the number of nontrivial $\ell$-colorings of $\cL$ is $\ell^{m_\ell(\cL)}$, where $m_\ell(\cL)$ is the rank of the nullspace of any maximal codimension $1$ submatrix of the coloring matrix $A(-1)$ modulo $\ell$. We refer to the quantity $m_\ell(\cL)$ as the \emph{$\ell$-coloring rank} of $\cL$.

\begin{proposition}\label{coloring rank propn}
    Let $\cL$ be a link and $\bar{\Delta}_\cL$ denote the reduction of the Alexander polynomial modulo $\ell$. Assume that $\bar{\Delta}_\cL\neq 0$. If $r_\ell(\Delta_{\cL})$ be the order of the zero of $\bar{\Delta}(\cL)$ at $t=-1$. Then the number of non-trivial $\ell$-colorings of $\cL$ is at most $\ell^{r_\ell(\Delta_{\cL})}$. In other words, $m_\ell(\cL)\leq r_\ell(\cL)$.
\end{proposition}
\begin{proof}
    Let $R$ be the ring $\F_\ell[t]_{(t-1)}$ obtained by localizing $\F_\ell[t]$ modulo the maximal ideal $(t-1)$ and let $\pi:=(t-1)$ be a uniformizer in $R$. Then the Smith normal form of the Alexander matrix $A(t)$ over $R$ is of the form $(\pi^{n_1}, \pi^{n_2}, \dots, \pi^{n_k})$ with $n_i\geq n_{i+1}$ and $r_\ell(\Delta_{\cL})=\sum_i n_i$. On the other hand, $m_p(\cL)$ is the number of $n_i$ which are positive. The inequality $m_\ell(\cL)\leq r_\ell(\cL)$ follows from this.
\end{proof}
\subsection{The double branched cover}
We recollect the fundamental properties of the oriented \emph{double branched cover} of \(S^3\) branched along a link \(\cL\subset S^3\).

A \emph{double branched cover} of \(S^3\) branched along \(\cL\) is a closed, oriented $3$--manifold $\cover(\cL)$ equipped with a surjective continuous map $p\colon \cover(\cL) \to S^3$ such that
\begin{enumerate}
  \item \(p\) is a $2$--fold covering map on \(\cover(\cL)\setminus p^{-1}(\cL)\),
  \item \(p^{-1}(\cL)\) is the fixed-point set of an orientation-preserving involution \(\tau\colon \cover(\cL)\to \cover(\cL)\). The map $\tau$ is called the \emph{deck involution}.
  \item For each point \(x\in \cL\) there is a neighborhood \(U\) of \(x\) in \(S^3\) with \(U\cong \mathbb{R}^3\) such that there is a homeomorphism \[p^{-1}(U)\xrightarrow{\sim} \{(z,t)\in\mathbb{C}\times\mathbb{R} : z^2=t\}\] which identifies $p^{-1}(U\cap \cL)$ with $\{(0,t):t\in \mathbb{R}\}$.
\end{enumerate}

\par The covering transformation (deck transformation) \(\tau\) is an orientation-preserving involution of \(\cover(\cL)\) whose fixed-point set is exactly the preimage of \(L\). The quotient \(\cover(\cL)/\langle \tau\rangle\) is topologically \(S^3\) with branch set \(\cL\).

\begin{proposition}
Let $\cL \subset S^3$ be a link. There exists a closed, oriented $3$-manifold $\Sigma_2(\cL)$ along $\cL$, together with a branched covering map 
\[
\pi:\Sigma_2(\cL) \longrightarrow S^3
\] 
of degree $2$, branched precisely along $\cL$. Moreover, $\Sigma_2(\cL)$ is unique up to orientation-preserving homeomorphism.
\end{proposition}

\begin{proof}
This is a standard result and we provide a sketch of the details for the benefit of the reader. This construction of the double branched cover proceeds in two steps, first over the complement of the link, and then locally around each component of the link. Connected double covers of the complement $S^3 \setminus \cL$ correspond bijectively to surjective homomorphisms $\pi_1(S^3 \setminus \cL) \twoheadrightarrow \mathbb{Z}/2\mathbb{Z}$. If a homomorphism sends a meridian of a knot component $\cK$ of $\cL$ to $0$, then the corresponding $2$-cover is unbranched along $\cK$. There is a canonical choice of such a homomorphism which sends each meridian of $\cL$ to the nontrivial element in $\mathbb{Z}/2\mathbb{Z}$. This determines a unique (up to isomorphism) connected two-sheeted covering space of $S^3 \setminus \cL$ which does not extend to an unbrached cover of any of the knot components.

\par To extend this cover over the link itself, one considers a tubular neighborhood of each component $K \subset \cL$, identified with $D^2 \times S^1$. Locally, the branching is modeled on the map $z \mapsto z^2$ from the disk $D^2$ to itself, a standard construction in complex topology. Gluing these local models to the double cover of the complement yields a closed, oriented $3$--manifold.  

\par The uniqueness of $\Sigma_2(\cL)$ follows from the uniqueness of the homomorphism of the link group to $\Z/2\Z$ which maps each meridian to $1$.
\end{proof}

\par We present an explicit construction of $\Sigma_2(\cL)$. Let $F$ be an oriented compact Seifert surface for $\cL$. By definition, this is a connected, oriented surface embedded in $S^3$ with boundary equal to $\cL$. Cutting the three-sphere along $F$ yields a three-manifold $M$ whose boundary consists of two disjoint copies of $F$, which we denote by $F^+$ and $F^-$. To build the double cover, we take two copies of $M$ and glue them along their boundary surfaces, but with the convention that $F^+$ from the first copy is glued to $F^-$ from the second, and $F^-$ from the first copy is glued to $F^+$ from the second. This gives the double branched cover $\cover(\cL)$. The covering map to $S^3$ is obtained by collapsing the two copies of $M$ back onto the original complement of the Seifert surface, and the branch set is the link $\cL$ along which the cutting was performed. The natural deck transformation of this cover is given by interchanging the two copies of $M$, which reflects the symmetry inherent in the construction.


\begin{proposition}\label{order of H1}
Let $K\subset S^3$ be a knot, and let $\Sigma_2(K)$ denote the oriented double cover of $S^3$ branched along $K$. Then $H_1(\Sigma_2(K);\mathbb{Z})$ is finite, and 
\[
|H_1(\Sigma_2(K);\mathbb{Z})| = |\Delta_K(-1)|,
\]
where $\Delta_K(t)$ is the normalized Alexander polynomial of $K$.
\end{proposition}

\begin{proof}
Let $X = S^3 \setminus K$ and $\op{G}_K = \pi_1(X)$. The abelianization map 
\[
\op{G}_K \to H_1(X;\mathbb{Z}) \cong \langle t \rangle \cong \mathbb{Z}
\]
defines the infinite cyclic cover $\widetilde{X}\to X$, and the Alexander module
$A_K := H_1(\widetilde{X};\mathbb{Z})$ is a finitely generated torsion $\mathbb{Z}[t,t^{-1}]$–module.  
If $M(t)$ is a presentation matrix for $A_K$, then the normalized Alexander polynomial $\Delta_K(t)$ generates the order ideal of its torsion submodule, i.e. it is the gcd of the $(n-1)\times(n-1)$ minors of $M(t)$.

Let $\varepsilon:\pi\to\mathbb{Z}/2$ send a meridian to the nontrivial element, and let $X_2$ be the corresponding unbranched double cover.  Passing to abelianizations gives
\[
H_1(X_2;\mathbb{Z}) \cong A_K \otimes_{\mathbb{Z}[t,t^{-1}]} 
\frac{\mathbb{Z}[t,t^{-1}]}{(t+1)}.
\]
Evaluating the presentation matrix $M(t)$ at $t=-1$ thus gives a presentation of $H_1(X_2;\mathbb{Z})$ by the integer matrix $M(-1)$.

The branched cover $\Sigma_2(K)$ is obtained from $X_2$ by filling in the lifts of meridians, which algebraically corresponds to imposing $t=-1$ on $A_K$. Hence
\[
H_1(\Sigma_2(K);\mathbb{Z}) \cong A_K/(t+1)A_K.
\]
Since $A_K$ is torsion, this group is finite, and its order equals $|\Delta_K(-1)|$ by the definition of the order ideal.  Equivalently, $\det M(-1)=\pm\Delta_K(-1)$.
\end{proof}

\begin{remark}
For links with more than one component, the situation is richer: the first homology of \(\cover(\cL)\) can have infinite rank depending on linking numbers. For an $\ell$-component link \(L\) the covering corresponding to sending all meridians to \(-1\) still exists, but \(H_1(\cover(\cL);\Z)\) may contain free summands; there are combinatorial formulas (in terms of linking numbers and Seifert matrices) which describe it.
\end{remark}

We describe the double (2--fold) branched cover of \(S^3\) branched along the torus knot \(T_{p,q}\) (with \(p,q\ge2\) coprime).  The main identification is
\[
\Sigma_2(S^3, T_{p,q}) \cong \Sigma(2,p,q),
\]
the Seifert fibered manifold with exceptional fibres of orders \(2,p,q\).  We give three complementary constructions (cut--and--paste along a Seifert surface and quotient/symmetry method, compute basic invariants, and work through central examples (the trefoil and the Poincaré homology sphere).

Fix coprime integers \(p,q\ge2\). Let \(T_{p,q}\subset S^3\) denote the (oriented) torus knot of type \((p,q)\) realized as a simple closed curve on the standardly embedded torus in \(S^3\) winding \(p\) times meridionally and \(q\) times longitudinally (or vice versa depending on conventions).

\begin{theorem}\label{thm:main}
The oriented double branched cover of \(S^3\) along the torus knot \(T_{p,q}\) is homeomorphic to the Brieskorn manifold \(\Sigma(2,p,q)\). Equivalently, it is the oriented Seifert fibered $3$--manifold with base \(S^2\) and three exceptional fibres of orders \(2,p,q\).
\end{theorem}

We will explain the meaning of \(\Sigma(2,p,q)\) and give proof and interpretation of Theorem \ref{thm:main}. A Seifert fibered manifold with base \(S^2\) and three exceptional fibres of multiplicities \((\alpha_1,\alpha_2,\alpha_3)\) is obtained by taking an \(S^1\)-bundle over \(S^2\) with three marked orbifold points and performing the standard Seifert filling data. We write such a manifold (up to orientation conventions) as a Seifert manifold with normalized Seifert invariants
\[
M\big(b;(\alpha_1,\beta_1),(\alpha_2,\beta_2),(\alpha_3,\beta_3)\big),
\]
where \(\alpha_i\) are the orders of the exceptional fibres and the \(\beta_i\) encode the slope/frame choices. The Brieskorn manifold \(\Sigma(2,p,q)\) is diffeomorphic to the Seifert fibered space with exceptional fibres of orders \(2,p,q\) and an explicitly computable normalized Seifert invariant. For the remainder we focus on the case \((a,b,c)=(2,p,q)\). Let \(F\) be a standard oriented Seifert surface for \(T_{p,q}\) obtained from the embedded torus: for the torus knot there is a natural once-punctured torus (if \(p,q>1\)) or an explicit Seifert surface of genus \(\frac{(p-1)(q-1)}{2}\) obtained by plumbing bands. Cut \(S^3\) along \(F\). The double cover is obtained by taking two copies of the cut manifold and gluing the two boundary copies of \(F\) with the boundary components interchanged. For torus knots this glued manifold carries a free \(S^1\)–action away from the branch set and the quotient by that \(S^1\)-action is an orbifold \(S^2\) with three cone points of orders \(2,p,q\). Hence the upstairs manifold is a Seifert fibered manifold with exactly three exceptional fibres of multiplicities \(2,p,q\), i.e. \(\Sigma(2,p,q)\). 
\par Now let us look at homology of the double branched cover. 
The first homology of the double branched cover of a knot equals the determinant of the knot:
\[
\big|H_1(\Sigma_2(S^3,K);\Z)\big| = |\Delta_K(-1)|.
\]
Applying this to torus knots gives the order of \(H_1(\Sigma(2,p,q);\Z)\) (finite in all these cases except where obvious free summands appear for multi-component branch sets). A few illustrative cases:
\begin{itemize}
  \item If one of \(p,q\) equals \(2\) (so the torus knot is \(T_{2,q}\), i.e. a 2--bridge knot), then \(\Sigma(2,2,q)\) is a lens space \(L(q,r)\) for an appropriate \(r\). In particular the trefoil \(T_{2,3}\) has double branched cover \(L(3,1)\).
  \item If both \(p\) and \(q\) are odd (and coprime, as required for a torus knot), then \(\Sigma(2,p,q)\) is an integral homology sphere. For example \(\Sigma(2,3,5)\) is the Poincaré homology sphere (the double branched cover of the \((3,5)\)--torus knot).
\end{itemize}


\begin{example}[Trefoil: \(T_{2,3}\)]
The right-handed trefoil is \(T_{2,3}\). Its double branched cover is the lens space \(L(3,1)\). This fits the general picture: when one multiplicity is \(2\) and the other is \(q\), the Brieskorn manifold \(\Sigma(2,2,q)\) (after normalization) is a lens space of order \(q\).
\end{example}
\begin{example}[General]
If \((p,q)\) are both odd and coprime then \(\Sigma(2,p,q)\) is a Seifert fibered integral homology sphere; when one of \(p,q\) equals \(2\) the cover is a lens space; intermediate arithmetic of \((2,p,q)\) controls the orbifold geometry.
\end{example}

\subsection{Coloring torus knots}

\par We study coloring invariants associated with torus knots $T_{p, q}$. First consider the case when $p$ and $q$ are coprime. It is well known that the Alexander polynomial of $T_{p, q}$ is given by 
\[\Delta_{p,q}(t)=\Delta(T_{p, q})(t)=\frac{(t^{pq}-1)(t-1)}{(t^p-1)(t^q-1)}.\]
The determinant of $T_{p, q}$ is $\Delta_{p, q}(-1)$. We note that when both $p$ and $q$ are odd, $\op{det}(T_{p,q})=1$ and hence $T_{p, q}$ is not $\ell$-colorable for any odd prime $\ell$. On the other hand, if $p$ is odd and $q$ is even, we find that 
\[\op{det}(T_{p,q})=\Delta_{p,q}(-1)=\lim_{t\rightarrow -1} \frac{\frac{d}{dt}\left((t^{pq}-1)(t-1)\right)}{\frac{d}{dt}\left((t^p-1)(t^q-1)\right)}=\frac{-2pq}{-2q}=p\] by an application of L Hopital's rule. Thus in this case, $T_{p, q}$ is $\ell$-colorable if and only if $\ell|m$. Similarly, if $p$ is even and $q$ is odd, then $T_{p,q}$ is $\ell$-colorable if and only if $\ell|n$.
In light of Proposition \ref{order of H1}, we find that $H_1(\Sigma_2(T_{p,q}))$ has nontrivial $\ell$-torsion if and only if $T_{p,q}$ is $\ell$-colorable. 

\section{Distribution of zeros of Alexander polynomials}
\subsection{Distribution results for torus knots}
\par Many calculations in this article can be reduced to explicit relations among the roots of Alexander polynomials, which in this section are shown to be equidistributed on the unit circle $\mathbb{T}:=\{z\in \mathbb{C}\mid |z|=1\}$. Let $p$ and $q$ be positive coprime integers and $T_{p,q}$ the associated torus knot. In particular, both $p$ and $q$ cannot be even. Recall that the genus of $T_{p,q}$ is given by $g(p, q)=\frac{(p-1)(q-1)}{2}$. We consider the family of all torus knots indexed by the following height function 
\[\op{ht}(p,q):=\op{max}\{p,q\}.\] Given a positive real number $X$, set 
\[\mathcal{T}_1(X):= \{(p,q)\mid p, q\text{ positive coprime integers }\op{ht}(p, q)\leq X\}.\]It is clear that $\mathcal{T}_1(X)$ is finite. 

\begin{remark}
    A few remarks are in order. One may enlarge the family by allowing all non-zero coprime pairs $(p,q)$, rather than restricting to positive integers. The resulting asymptotic questions are essentially the same, though the book-keeping becomes slightly more involved. One could replace the height bound $\max\{p,q\}\le X$ by a geometric bound, for instance by ordering torus knots according to their genus. This leads to a different but closely related counting problem, and it would be interesting to compare the asymptotics.

\end{remark}

\par Let \(\mu:\mathbb{N}\to \{-1,0,1\}\) be the M\"obius function, defined by:
\[
\mu(n)=
\begin{cases}
1 & \text{if } n=1,\\[4pt]
(-1)^k & \text{if } n \text{ is a product of } k \text{ distinct primes},\\[4pt]
0 & \text{if } n \text{ is divisible by the square of a prime}.
\end{cases}
\]
One of its fundamental properties is the classical identity
\begin{equation}\label{mobius identity}
\sum_{d\mid n} \mu(d)=
\begin{cases}
1 & \text{if } n=1,\\
0 & \text{if } n>1,
\end{cases}
\end{equation}
which expresses the fact that \(\mu\) is the Dirichlet inverse of the constant function \(1\). 
 \par Given real-valued functions $f(X)$, $g(X)$ and a third real valued function $h(X)$ which is eventually positive, then we write $f(X)=g(X)+O\left(h(X)\right)$ to mean that \[\lim_{X\rightarrow \infty} \frac{|f(X)-g(X)|}{h(X)}=0.\]
\noindent Write $f(X)\sim h(X)$ to mean that 
\[\lim_{X\rightarrow \infty} \frac{f(X)}{h(X)}=1.\]
\noindent Note that $\zeta(2)=\sum_{n=1}^\infty n^{-2}=\pi^2/6$.
\begin{lemma}\label{lemma on T_1 asymptotic}
    With respect to notation above, 
   \begin{equation}\label{T(X) asymptotic}
       \#\mathcal{T}_1(X) = \frac{1}{\zeta(2)}X^2 + O(X\log X).
   \end{equation}
In particular,
\[
\#\mathcal{T}_1(X) \sim \frac{1}{\zeta(2)}X^2 \qquad \text{as } X\to\infty.
\]
\end{lemma}

\begin{proof} Note that 
\[\begin{split}
\#\mathcal{T}_1(X)
&= \sum_{1\leq p\leq \lfloor X\rfloor } \sum_{1\leq q \leq \lfloor X\rfloor } \sum_{d\mid(p,q)} \mu(d),\\
&= \sum_{d\ge 1} \mu(d)\, \#\{(p,q)\mid 1\le p,q\le \lfloor X\rfloor,\, d\mid p,\, d\mid q\},\\
&=\sum_{d\le X} \mu(d)\, \lfloor X/d \rfloor^2,
\end{split}
\]
where the first equality follows from \eqref{mobius identity}. Write $\lfloor X/d \rfloor = \frac{X}{d} + \theta_d$, where $|\theta_d|\le 1$. Substituting this into the previous expression gives
\begin{align*}
\#\mathcal{T}_1(X)
&= \sum_{d\le X} \mu(d) \left(\frac{X^2}{d^2} + 2\frac{X}{d}\theta_d + \theta_d^2 \right) \\
&= X^2 \sum_{d\le X} \frac{\mu(d)}{d^2}
  + 2X\sum_{d\le X}\frac{\mu(d)\theta_d}{d}
  + \sum_{d\le X}\mu(d)\theta_d^2.
\end{align*}

We analyze each of the three sums separately. The main term is
\[
X^2 \sum_{d\le X}\frac{\mu(d)}{d^2} = \frac{1}{\zeta(2)}X^2 + O(X).
\]
\noindent Since $|\theta_d|\le 1$, we have
\[
\left|\sum_{d\le X}\frac{\mu(d)\theta_d}{d}\right| \le \sum_{d\le X}\frac{1}{d} = O(\log X),
\]
so this contributes at most $O(X\log X)$ to $\#\mathcal{T}_1(X)$. Finally, since $|\theta_d^2|\le 1$ and $|\mu(d)|\le 1$, we have
\[
\left|\sum_{d\le X}\mu(d)\theta_d^2\right|\le X,
\]
so this term contributes $O(X)$. Combining all three contributions, we obtain the asymptotic formula \eqref{T(X) asymptotic}.\end{proof}
Let $\mathcal{T}_1$ be the set of all pairs $(p,q)$, where $p$ and $q$ are coprime positive integers. Identify $(p, q)\in \mathcal{T}_1$ with the torus knot $T_{p, q}$ and thus $\mathcal{T}_1(X)$ is identified with the set of torus knots $T_{p,q}$ with height $\op{ht}(T_{p,q}):=\op{ht}(p,q)\leq X$. Let $\zeta_{m}:=\op{exp}(2\pi i /m)$; recall that the Alexander polynomial of $T_{p,q}$ is given by 
\begin{equation}\label{Delta_{p,q}(t) formula}\Delta_{p,q}(t)=\Delta_{T_{p, q}}(t)=\frac{(t^{pq}-1)(t-1)}{(t^p-1)(t^q-1)}=\prod_{k} (t-\zeta_{pq}^k),\end{equation} where $k$ runs over the integers $1\leq k\leq {pq}$ such that $k$ is not divisible by $p$ or $q$. When $p$ and $q$ are both prime numbers, $\Delta_{p,q}(t)$ is simply the cyclotomic polynomial $\Phi_{pq}(t)$. \par The roots are not quite evenly spaced on the unit circle, however are equidistributed in the limit with respect to the Haar measure, as we shall make precise. Given a subset $\mathcal{S}$ of $\mathcal{T}_1$ and $X>0$, set \[\mathcal{S}(X):=\mathcal{S}\cap \mathcal{T}_1(X)=\{T_{p, q}\in \mathcal{S}\mid \op{ht} (T_{p, q})\leq X\}.\] The \emph{density} of $\mathcal{S}$ is defined to be the following limit, provided it exists:
\[\mathfrak{d}(\mathcal{S}):=\lim_{X\rightarrow \infty} \left(\frac{\# \mathcal{S}(X)}{ \#\mathcal{T}_1(X)}\right)=\zeta(2)\lim_{X\rightarrow \infty} \left(\frac{\# \mathcal{S}(X)}{ X^2}\right).\] The second equality above follows from \eqref{T(X) asymptotic}. Set $\Omega_1$ to denote all pairs $(T_{p,q}, \alpha)$, where $p,q$ are coprime positive integers $\alpha$ is a root of the Alexander polynomial $\Delta_{p, q}(t)$. Given $u=(T_{p,q}, \alpha)\in \Omega_1$, take $\op{height}(u):=\op{height}(T_{p,q})=\op{max}\{p, q\}$. Then we set $\Omega_1(X)$ to denote the set of $u\in \Omega_1$ such that $\op{height}(u)\leq X$. 
\begin{proposition}
    With respect to notation above, 
\[
\#\Omega_1(X)=\frac{1}{4 \zeta(2)} X^4 + O(X^3\log X),
\]
and in particular
\[
\#\Omega_1(X)\sim \frac{1}{4 \zeta(2)}X^4 .
\]
\end{proposition}
\begin{proof}
    Note that 
    \[\# \Omega_1(X)=\sum_{1\leq p, q\leq \lfloor X\rfloor } (p-1)(q-1)\sum_{d|(p,q)} \mu(d).\]
\noindent Interchange the order of summation to obtain
\[
\#\Omega_1(X)=\sum_{d\ge1}\mu(d)\sum_{\substack{1\le p,q\le\lfloor X\rfloor\\ d\mid p,\ d\mid q}} (p-1)(q-1).
\]
For \(d>\lfloor X\rfloor\) the inner sum is zero, so restrict to \(d\le X\). Put \(p=dp'\), \(q=dq'\) with \(1\le p',q'\le \lfloor\frac{X}{d}\rfloor\) and thus
\[
\#\Omega_1(X)=\sum_{d\le X}\mu(d)\sum_{1\le p',q'\le \lfloor\frac{X}{d}\rfloor} (dp'-1)(dq'-1).
\]
Compute the inner double sum explicitly. Let
\[
S_1(N)=\sum_{n=1}^N n=\frac{N(N+1)}{2},\qquad
S_2(N)=\sum_{n=1}^N n^2=\frac{N(N+1)(2N+1)}{6}.
\]
Then
\begin{align*}
\sum_{1\le p',q'\le N} (dp'-1)(dq'-1)
&= d^2\sum_{1\le p',q'\le N} p'q' - d\sum_{1\le p',q'\le N} p' - d\sum_{1 \le p',q'\le N} q' + \sum_{1\le p',q'\le N} 1 \\
&= d^2\big(S_1(N)\big)^2 - 2d\,N\,S_1(N) + N^2.
\end{align*}
\noindent Therefore
\[
\#\Omega_1(X)=\sum_{d\le X}\mu(d)\Big( d^2\big(S_1(N)\big)^2 - 2d\,N\,S_1(N) + N^2\Big).
\]
\noindent Set \(N=\lfloor X/d\rfloor = X/d + \theta_d\) with \(|\theta_d|\le1\). Since
\[
S_1(N)=\frac{N(N+1)}{2}=\frac{N^2}{2}+\frac{N}{2},
\]
we expand the main (quadratic) piece:
\[
d^2\big(S_1(N)\big)^2
= d^2\frac{N^2(N+1)^2}{4}
= \frac{d^2 N^4}{4} + O(d^2 N^3).
\]
Substituting \(N=X/d+\theta_d\) gives the leading term
\[
\frac{d^2 N^4}{4} = \frac{X^4}{4d^2} + O\!\Big(\frac{X^3}{d}\Big).\] The remaining two terms satisfy the bounds
\[
-2dN S_1(N) + N^2 = O(d N^3) + O(N^2) = O\!\Big(\frac{X^3}{d}\Big) + O(X^2).
\]
Thus for each \(d\le X\) we have the pointwise estimate
\[
\sum_{1\le p',q'\le N}(dp'-1)(dq'-1)
= \frac{X^4}{4d^2} + O\!\Big(\frac{X^3}{d}\Big) + O(X^2).
\]
\noindent Insert this into the outer sum and use absolute bounds on \(\mu(d)\):
\[
\#\Omega_1(X)
= \frac{X^4}{4}\sum_{d\le X}\frac{\mu(d)}{d^2}
+ O\!\Big(X^3\sum_{d\le X}\frac{1}{d}\Big)
+ O\!\Big(X^2\sum_{d\le X}1\Big).
\]
Note that
\[
\sum_{d\le X}\frac{\mu(d)}{d^2}=\frac{6}{\pi^2}+O\!\Big(\sum_{d>X}\frac{1}{d^2}\Big)
=\frac{6}{\pi^2}+O\!\Big(\frac{1}{X}\Big).
\]
\noindent The harmonic sum satisfies \(\sum_{d\le X}1/d= \log X + \gamma + O(1/X)\), so
\[
X^3\sum_{d\le X}\frac{1}{d}=O(X^3\log X).
\]
\noindent Combining these estimates yields
\[
\#\Omega_1(X)
= \frac{X^4}{4}\Big(\frac{6}{\pi^2}+O\!\Big(\frac{1}{X}\Big)\Big)
+ O(X^3\log X)
+ O(X^3).
\]
Consequently
\[
\#\Omega_1(X)=\frac{1}{4 \zeta(2)} X^4 + O(X^3\log X).
\]

\end{proof}
\par Let $\mathbb{T}$ denote the unit circle in $\mathbb{C}$ and write $z\in \mathbb{T}$ as $z=\op{exp}(2\pi i \theta(z))$, where $\theta(z)\in [0, 1)$.

\begin{theorem}\label{section 3 main thm 1}
    The set of roots of the Alexander polynomials of torus knots is equidistributed on the unit circle in the following sense. Given any interval $[a,b]\subset [0, 1]$, 
    \begin{equation}\#\{(T_{p,q}, \alpha)\in \Omega_1(X)\mid \theta(\alpha)\in [a,b]\}=(b-a) \# \Omega_1(X)+O(X^2).\end{equation}
    In particular, one finds that 
    \[\#\{(T_{p,q}, \alpha)\in \Omega_1(X)\mid \theta(\alpha)\in [a,b]\}=\frac{(b-a)}{4 \zeta(2)} X^4 + O(X^3\log X)\]
\end{theorem}
\begin{proof}
The number of roots of \(\Delta_{p,q}\) whose argument \(\theta\) lies in \([a,b]\) equals the number of integers $k\in\{1,2,\dots,pq-1\}$ such that $\frac{k}{pq}\in[a,b]$, which are not divisible by \(p\) or by \(q\). Let \(N_{p,q}([a,b])\) denote this number. Then, we find that
\[
\begin{split}N_{p,q}([a,b])
= &\big(\lfloor b\,pq\rfloor-\lfloor a\,pq\rfloor\big)
- \#\{k\in\mathbb Z\cap(a pq,b pq]:\ p\mid k\}
- \#\{k\in\mathbb Z\cap(a pq,b pq]:\ q\mid k\}\\
+ &\#\{k\in\mathbb Z\cap(a pq,b pq]:\ pq\mid k\}.\\
\end{split}
\]
Counting the multiples gives
\[
\#\{k\in\mathbb Z\cap(a pq,b pq]:\ p\mid k\}=\lfloor bq\rfloor-\lfloor a q\rfloor,
\qquad
\#\{k\in\mathbb Z\cap(a pq,b pq]:\ q\mid k\}=\lfloor b p\rfloor-\lfloor a p\rfloor,
\]
and the last term is either \(0\) or \(1\). From these identities,
\[
N_{p,q}([a,b]) = (b-a)(p-1)(q-1) + E_{p,q},
\]
with \(|E_{p,q}|\ll 1\) uniformly in \(p,q\).
\par We now count over all torus knots with height \(\le X\). By definition the left side in the theorem equals
\[
\#\{(T_{p,q},\alpha)\in\Omega_1(X)\mid \theta(\alpha)\in[a,b]\}
= \sum_{1\le p,q\le\lfloor X\rfloor}( \sum_{d\mid(p,q)}\mu(d) )\; N_{p,q}([a,b]).
\]
Substituting the formula for \(N_{p,q}([a,b])\) gives
\[
\begin{aligned}
\#\{(T_{p,q},\alpha)\in\Omega_1(X)\mid \theta(\alpha)\in[a,b]\}
&= (b-a)\sum_{1\le p,q\le\lfloor X\rfloor}(p-1)(q-1)\sum_{d\mid(p,q)}\mu(d) \\
&\qquad + \sum_{1\le p,q\le\lfloor X\rfloor}\Big(\sum_{d\mid(p,q)}\mu(d)\Big)E_{p,q}.
\end{aligned}
\]
The first sum on the right is \((b-a)\,\#\Omega_1(X)\) by the definition of \(\#\Omega_1(X)\). For the second sum note that \(|E_{p,q}|\ll1\) and
\[
\sum_{1\le p,q\le\lfloor X\rfloor}\Big|\sum_{d\mid(p,q)}\mu(d)\Big|
= \#\mathcal{T}_1(X)\ll X^2.
\] This completes the proof.
\end{proof}

\subsection{Distribution results for the family of torus links}
\par Next we study a similar root equidistribution question for torus links associated to all pairs of positive integers $(p,q)$ (and not just torus knots). Given such a pair $(p,q)$, set \(d:=\gcd(p,q)\), \(p'=p/d\), \(q'=q/d\), and
$L:=\operatorname{lcm}(p,q)=\frac{pq}{d}=d\,p'\,q'$. When $d\geq 2$, the multivariable Alexander polynomial is given by 
\begin{equation}\label{MAP eqn}\Delta_{T_{p,q}}(X_1, \dots, X_d)=\frac{\left((X_1\dots X_d)^{L}-1\right)^d}{\left((X_1\dots X_d)^{p'}-1\right)\left((X_1\dots X_d)^{q'}-1\right)}\end{equation}
\noindent (cf. \cite[section 10.1]{Milnor}).
When $d\geq 2$, the single variable Alexander polynomial admits the closed form
\[
\Delta_{p,q}(t)
:=(t-1)\Delta_{T_{p,q}}(t, t, \dots, t)= \frac{(t^{L}-1)^{d}(t-1)}{(t^{p}-1)(t^{q}-1)}=\prod_{r\mid L}\Phi_r(t)^{M_r},
\]
where $\Phi_r(t)$ is the $r$-th cyclotomic polynomial and $M_r:=\operatorname{ord}_{\Phi_r}(\Delta_{p,q})$. Note that when $d=1$, the formula \eqref{Delta_{p,q}(t) formula} for $\Delta_{p,q}(t)$ matches the above.

\par Given $X>0$, let $\mathcal{T}(X)$ be the set of all pairs of positive integers $(p,q)$ with $\op{ht}(p,q)\leq X$, i.e., $p, q\leq X$. As is convention, identify $(p,q)$ with the torus link $T_{p,q}$. Also denote by $\Omega(X)$ the multiset of all pairs $(T_{p,q}, \alpha)$ where $T_{p,q}\in \mathcal{T}(X)$ and $\alpha$ is a root of $\Delta_{p,q}(t)$. Given a closed interval $[a,b]$, let $\Omega_{[a,b]}(X)$ be the multiset of pairs $(T_{p,q}, \alpha)\in \Omega(X)$ for which $\alpha=\op{exp}(2\pi i \theta)$ with $\theta\in [a,b]$. 

\begin{theorem}\label{section 3 main thm 2}
    With respect to notation above, 
\begin{equation}\label{Omega asymptotic 1}\#\Omega_{[a,b]}(X)=\frac{(b-a)}{4} X^4+O(X^3)\end{equation}\noindent and\begin{equation}\label{Omega asymptotic 2}\lim_{X\rightarrow \infty} \frac{\#\Omega_{[a,b]}(X)}{\#\Omega(X)}.\end{equation}
\end{theorem}
\begin{proof}
\par We write each cyclotomic factor as \(t^n-1=\prod_{r\mid n}\Phi_r(t)\). Comparing exponents of the cyclotomic polynomials \(\Phi_r\) in the expression for \(\Delta_{p,q}\) shows that for each \(r\ge1\)
\[
M_r= d \mathbf{1}_{r\mid L}-\mathbf{1}_{r\mid p}-\mathbf{1}_{r\mid q} + \mathbf{1}_{r=1},
\]
where \(\mathbf{1}_{P}\) denotes the indicator of the predicate \(P\). In particular:
\begin{itemize}
    \item If $r$ does not divide $L$ then $M_r=0$.
 \item If $r$ divides $L$ but does not divide $p$ or $q$, then $M_r=d$.
  \item If \(r>1\) divides \(d\) then
  \(M_r=(d-2)\).
  \item If \(r>1\) divides $p$ (resp. $q$) but does not divide $q$ (resp. $p$), then \(M_r=(d-1)\).
  \item For the linear factor \(r=1\) (i.e. \(t-1\)) the exponent is \(M_r=(d-1)\).
\end{itemize}
In particular every root of \(\Delta_{p,q}\) is a root of unity whose order divides \(L\), and the multiplicity of any given primitive \(r\)-th root of unity is bounded above by \(d\). 
\par For a pair \((p,q)\) let
\[R_{p,q}:=Ld+1-p-q=(p-1)(q-1)\] be the total number of roots of \(\Delta_{p,q}\) counted with multiplicity. Thus one has that
\[
\#\Omega(X):=\sum_{1\le p,q\le\lfloor X\rfloor} R_{p,q}=\sum_{1\le p,q\le\lfloor X\rfloor} (p-1)(q-1)=\frac{1}{4}\left(\lfloor X\rfloor(\lfloor X\rfloor-1)\right)^2=\frac{1}{4}X^4+O(X^3).
\]
Since every root is a root of unity of order dividing \(L\),
we can enumerate them as \(e^{2\pi i k/L}\) for \(k=1,\dots,L-1\) with certain multiplicities. Let $S_r$ denote the multiset of primitive $r$-th roots of unity, each counted with multiplicity $M_r$.  Then the full multiset of zeros of $\Delta_{p,q}$ (with multiplicity) is the disjoint union $\bigsqcup_{r\mid L} S_r$, and 
\[\sum_{r\mid L} M_r\varphi(r)
=R_{p,q}=(p-1)(q-1),
\]
where $\varphi$ is Euler's totient function.
\par For any interval $[a,b]\subset[0,1]$, define
\[
N_r([a,b]) := \#\{1\le k\le r : \gcd(k,r)=1,\ \tfrac{k}{r}\in[a,b]\}.
\]
Letting $\mathbf{1}_{[a,b]}$ be the indicator functor for containment in $[a,b]$ defined as follows:
\[\mathbf{1}_{[a,b]}(x):=\begin{cases}
    & 1\text{ if }x\in [a,b];\\
    & 0\text{ if }x\notin [a,b].
\end{cases}\]

One finds that 
\[\begin{split}N_r([a,b])=&\sum_{k=1}^r \mathbf{1}_{[a,b]}(k/r)\sum_{d|(k,r)} \mu(d)\\
=& \sum_{d|r} \mu(d)\sum_{k=1}^{r/d} \mathbf{1}_{[a,b]}(kd/r)\\
=& \sum_{d|r} \mu(d)\sum_{k=1}^{r/d} \mathbf{1}_{[a,b]}(k/(r/d))\\
=& \sum_{d|r} \mu(d)\sum_{k=1}^{r/d} \mathbf{1}_{[ra/d,rb/d]}(k)\\
=& (b-a)\sum_{d|r} \mu(d)\frac{r}{d}+O\left(\sum_{d|r} 1\right)\\
=& (b-a)\varphi(r) + O(d(r))
\end{split}\]
where $d(r)$ denotes the number of positive divisors of $r$. In particular, for every $\varepsilon>0$, there exists a constant $C_\varepsilon>0$
such that
\[
N_r([a,b]) = (b-a)\varphi(r) + \delta_r, \qquad |\delta_r|\le C_\varepsilon r^\varepsilon.
\]
\noindent Multiplying by the multiplicity $M_r$ of the factor $\Phi_r(t)$ in
$\Delta_{p,q}(t)$ yields
\[
\#\{\xi\in S_r\mid \theta(\xi)\in[a,b]\}
= M_r N_r([a,b])
= M_r\big((b-a)\varphi(r)+\delta_r\big),
\qquad |\delta_r|\ll d(r).
\]
\noindent Summing over all divisors $r\mid L$, we obtain
\begin{equation}\label{eq:local-distribution}
\begin{split}
& \sum_{r|L}\# \{\xi\in S_r\mid \theta(\xi)\in [a,b]\}\\
=&\sum_{r|L} M_r N_{p,q}([a,b])\\
=&(b-a)\sum_{r\mid L} M_r\varphi(r)
+\sum_{r\mid L} M_r\delta_r
=(b-a)R_{p,q}+E_{p,q},
\end{split}
\end{equation}
where the total error term is
\[
E_{p,q}=\sum_{r\mid L}M_r\delta_r\leq d\sum_{r|L} \delta_r\ll_{\epsilon} d L^{\epsilon}.
\]
\noindent Recall that $p=dp'$, $q=dq'$ with $(p',q')=1$. Thus writing $L=d p'q'$ with $p',q'\le X/d$, one finds that
\[\begin{split}
&\sum_{1\le p,q\le X} |E_{p,q}| \\
\ll_{\epsilon} &\sum_{d\le X} d^{1+\epsilon} \sum_{p',q'\le X/d} (p'q')^\epsilon\\
\leq &\sum_{d\le X} d^{1+\epsilon} \left(\frac{X}{d}\right)^{2+2\epsilon} \\
\leq & X^{2+2\epsilon}\sum_{d\le X} \frac{1}{d^{1+\epsilon}}
\ll_{\epsilon} X^{2+\epsilon}.
\end{split}
\]
Thus we find that 
\[\begin{split}
    \# \Omega_{[a,b]}(X)= &
    \sum_{1\leq p,q\leq X}\sum_{r|L(p,q)}\# \{\xi\in S_r\mid \theta(\xi)\in [a,b]\}\\
    = &(b-a) \sum_{1\leq p,q\leq X} R_{p,q}+\sum_{1\leq p,q\leq X}E_{p,q}\\
    = &(b-a)\#\Omega(X)+O_{\epsilon}(X^{2+\epsilon}).
\end{split}\]
Since $\#\Omega(X)=\frac{1}{4} X^4+O(X^3)$, \eqref{Omega asymptotic 1} and \eqref{Omega asymptotic 2} follow from the above.
\end{proof}
\subsection{Weak convergence of measures}
We interpret Theorems \ref{section 3 main thm 1} and \ref{section 3 main thm 2} as weak convergence of measures. We set
\[
\begin{split}
&\mu_X' := \frac{1}{\#\Omega_1(X)}\sum_{(T_{p,q},\alpha)\in\Omega_1(X)} \delta_{\alpha}\\
&\mu_X := \frac{1}{\#\Omega(X)}\sum_{(T_{p,q},\alpha)\in\Omega(X)} \delta_{\alpha}
\end{split}
\]
\noindent on $\mathbb{T}$, where $\delta_{\alpha}$ denotes the Dirac probability measure supported at the point $\alpha\in \mathbb{T}$. Let $\mu_{\mathbb{T}}$ be the normalized Haar measure on $\mathbb{T}$. Theorem \ref{section 3 main thm 1} (resp. Theorem \ref{section 3 main thm 2}) implies that $\mu_X$ (resp. $\mu_X'$) converges weakly to $\mu_{\mathbb{T}}$ as $X\to\infty$. 
\begin{theorem}\label{measure theoretic convergence theorem}
for every Riemann-integrable function $f:\mathbb{T}\to\mathbb{C}$ one has
\[
\lim_{X\rightarrow \infty}\int_{\mathbb{T}}f(z)\,d\mu_X^*(z)
= \int_{\mathbb{T}} f(z)\,d\mu_{\mathbb{T}}(z),
\]
where $\mu_X^*$ is either $\mu_X$ or $\mu_X'$.
\end{theorem}
\begin{proof}
\par We prove the result for $\mu_X$, the same argument applies to $\mu_X'$ as well. We proceed by successive approximation, starting with characteristic functions of intervals. For $[a,b]\subset [0,1]$, consider the function
\[
\mathbf{1}_{[a,b]}(e^{2\pi i\theta})
= \begin{cases}
1, & \theta\in[a,b],\\
0, & \text{otherwise.}
\end{cases}
\]
Then
\[
\int_{\mathbb{T}}\mathbf{1}_{[a,b]}(z)\,d\mu_X(z)
= \mu_X\big(\{e^{2\pi i\theta}\mid \theta\in[a,b]\}\big)
= \frac{\#\Omega_{[a,b]}(X)}{\#\Omega(X)}.
\]
By Theorem~\ref{section 3 main thm 1}, 
\[
\lim_{X\rightarrow \infty}\frac{\#\Omega_{[a,b]}(X)}{\#\Omega(X)}=(b-a)=\int_{\mathbb{T}}\mathbf{1}_{[a,b]}(z)\,d\mu_{\mathbb{T}}(z)
\]
and hence
\[
\lim_{X\rightarrow \infty}\int_{\mathbb{T}}\mathbf{1}_{[a,b]}(z)\,d\mu_X(z)
= \int_{\mathbb{T}} \mathbf{1}_{[a,b]}(z)\,d\mu_{\mathbb{T}}(z).
\]
Next, let $g$ be a step function on $\mathbb{T}$, that is, a finite linear combination of characteristic functions of disjoint intervals:
\[
g(e^{2\pi i\theta}) = \sum_{j=1}^m c_j\,\mathbf{1}_{[a_j,b_j]}(e^{2\pi i\theta}),
\qquad c_j\in\mathbb{C}.
\]
Then by linearity of the integral,
\[
\lim_{X\rightarrow \infty}\int_{\mathbb{T}}g(z)\,d\mu_X(z)
= \int_{\mathbb{T}} g(z)\,d\mu_{\mathbb{T}}(z).
\]
Let now $f:\mathbb{T}\to\mathbb{C}$ be a continuous function. Since $\mathbb{T}$ is compact, $f$ is uniformly continuous and bounded. For every $\varepsilon>0$, there exists a step function $g$ such that
\[
\|f-g\|_\infty < \varepsilon.
\]
Then
\[
\Bigl|\int_{\mathbb{T}} f(z)\,d\mu_X(z)
- \int_{\mathbb{T}} f(z)\,d\mu_{\mathbb{T}}(z)\Bigr|
\le
\Bigl|\int_{\mathbb{T}} (f-g)(z)\,d\mu_X(z)\Bigr|
+ \Bigl|\int_{\mathbb{T}} (f-g)(z)\,d\mu_{\mathbb{T}}(z)\Bigr|
+ \Bigl|\int_{\mathbb{T}} g(z)\,d\mu_X(z)
- \int_{\mathbb{T}} g(z)\,d\mu_{\mathbb{T}}(z)\Bigr|.
\]
Since both $\mu_X$ and $\mu_{\mathbb{T}}$ are probability measures, we have
\[
\Bigl|\int_{\mathbb{T}} (f-g)(z)\,d\mu_X(z)\Bigr|
\le \|f-g\|_\infty < \varepsilon,
\qquad
\Bigl|\int_{\mathbb{T}} (f-g)(z)\,d\mu_{\mathbb{T}}(z)\Bigr| < \varepsilon.
\]
The last term tends to $0$ as $X\to\infty$. Hence for $X$ large enough,
\[
\Bigl|\int_{\mathbb{T}} f(z)\,d\mu_X(z)
- \int_{\mathbb{T}} f(z)\,d\mu_{\mathbb{T}}(z)\Bigr| < 3\varepsilon.
\]
Since $\varepsilon>0$ is arbitrary, this proves that
\[
\lim_{X\rightarrow \infty}\int_{\mathbb{T}} f(z)\,d\mu_X(z)
= \int_{\mathbb{T}} f(z)\,d\mu_{\mathbb{T}}(z).
\]
\end{proof}
\subsection{Multiplicity statistics for torus knots}
\par Throughout this subsection \(p,q\) are coprime positive integers and
\(T_{p,q}\) denotes the associated torus \emph{knot} (so \(d=\gcd(p,q)=1\)).
Recall the single-variable Alexander polynomial
\[
\Delta_{p,q}(t)=\frac{(t^{pq}-1)(t-1)}{(t^p-1)(t^q-1)}
\]
and its cyclotomic factorization
\[
\Delta_{p,q}(t)=\prod_{r\ge1} \Phi_r(t)^{M_r(p,q)},
\]
where \(\Phi_r\) is the \(r\)-th cyclotomic polynomial and \(M_r(p,q)\) is the integer exponent (zero for all but finitely many \(r\)). For \(r\ge2\) we have
\[
M_r(p,q)=
\begin{cases}
1, &\text{if } r\mid pq \text{ but } r\nmid p,\ r\nmid q,\\[4pt]
0, &\text{otherwise.}
\end{cases}
\]
\noindent Moreover \(M_1(p,q)=1\) (the linear factor \(t-1\) appears with exponent \(1\)).
\par Fix \(r\ge2\). For each coprime pair \((p,q)\) with \(1\le p,q\le X\) define
the indicator
\[
\mathbf{1}_r(p,q) := \begin{cases}1,& M_r(p,q)=1,\\ 0,& M_r(p,q)=0.\end{cases}
\]
We want the asymptotic frequency
\[
\mathcal{F}_r(X) := \frac{1}{\#\mathcal{T}_1(X)}\sum_{(p,q)\in\mathcal{T}_1(X)} \mathbf{1}_r(p,q),
\]
where \(\mathcal{T}_1(X)=\{(p,q):1\le p,q\le X,\ \gcd(p,q)=1\}\) satisfies the asymptotic 
\[
\#\mathcal{T}_1(X) \sim \frac{1}{\zeta(2)}X^2,\] according to Lemma \ref{lemma on T_1 asymptotic}. The following
gives a precise limiting value.

\begin{theorem}\label{thm:avg-Mr}
Fix an integer \(r \ge 2\), and let \(\omega(r)\) denote the number of distinct prime divisors of \(r\). Then
\[
\lim_{X\to\infty} \mathcal{F}_r(X) = \frac{2^{\omega(r)} - 2}{r}.
\]
Equivalently, the average multiplicity of primitive \(r\)-th roots of unity among the zeros of the Alexander polynomial \(\Delta_{p,q}(t)\), averaged over all torus knots \(T_{p,q}\) with \(|p|,|q|\le X\), tends to
\[
\frac{\varphi(r)\bigl(2^{\omega(r)} - 2\bigr)}{r}.
\]
\end{theorem}

\begin{proof}
We begin by factoring the integer \(r\) as $r = \prod_{i=1}^{s} p_i^{e_i}$,
where \(p_1,\dots,p_s\) are distinct primes and \(s=\omega(r)\) is the number of distinct prime factors of \(r\). Suppose that $(p,q)$ is a pair such that $\op{gcd}(p,q)=1$ and \(r\mid pq\) but \(r\nmid p\) and \(r\nmid q\). This means that for every prime power factor \(p_i^{e_i}\) of \(r\), it divides either \(p\) or \(q\), but not both, and moreover not all of them divide the same one of \(p\) or \(q\). Thus, for the pair \((p,q)\), we must assign each prime power divisor \(p_i^{e_i}\) of \(r\) to one of the two integers \(p\) or \(q\), with the restriction that both \(p\) and \(q\) receive at least one factor. This amounts to choosing a partition of the index set \(\{1,\dots,s\}\) into two nonempty disjoint subsets
\[
S_1, S_2 \subset \{1,\dots,s\}, \qquad S_1\cup S_2 = \{1,\dots,s\}, \quad S_1,S_2\ne\emptyset.
\]
There are \(2^s - 2\) such ordered partitions.  

For a fixed partition, define
\[
a = \prod_{i\in S_1} p_i^{e_i}, \qquad b = \prod_{i\in S_2} p_i^{e_i}.
\]
Then \(a,b\ge2\), \(\gcd(a,b)=1\), $a|p$, $b|q$ and \(ab=r\). Note that \(M_r(p,q)=1\) if and only if there exists such a pair \((a,b)\) with \(ab=r\), \(\gcd(a,b)=1\), \(a,b>1\), and \(a\mid p\), \(b\mid q\).
\par Fix one such ordered factorization \(r=ab\) with \(\gcd(a,b)=1\), \(a,b>1\).  
We wish to estimate the number of coprime pairs \((p,q)\) with \(p,q\le X\), \(a\mid p\), \(b\mid q\). The number of integers \(p\) with \(p\le X\) divisible by \(a\) is asymptotic to \(X/a\), and likewise the number of integers \(q\) with \(q\le X\) divisible by \(b\) is asymptotic to \(X/b\). Hence the total number of such pairs \((p,q)\) (without the coprimality restriction) is asymptotically $\frac{X^2}{ab} = \frac{X^2}{ab}$. Among all integer pairs \((p,q)\) with \(p,q\le X\), the proportion of coprime pairs tends to \(1/\zeta(2)\). Therefore, the number of coprime pairs with \(p,q\le X\) is asymptotic to \(X^2/\zeta(2)\).  Since $a$ and $b$ are themselves coprime, the asymptotic number of coprime pairs \((p,q)\) with \(a\mid p\), \(b\mid q\) is \[\frac{X^2}{ab\zeta(2)}=\frac{X^2}{r\zeta(2)}\]. 
\par Summing over all \(2^s-2\) possible ordered nontrivial partitions of the set of prime factors, we obtain
\[
\lim_{X\to\infty} \mathcal{F}_r(X)
= \sum_{\substack{(a,b):\; ab=r\\ \gcd(a,b)=1,\, a,b>1}} \frac{1}{ab}
= \frac{2^s - 2}{r}
= \frac{2^{\omega(r)} - 2}{r}.
\] 
Finally, note that for each \(r\), there are exactly \(\varphi(r)\) primitive \(r\)-th roots of unity. Since each such root contributes equally to the multiplicity count, the limiting average multiplicity of \emph{primitive} \(r\)-th roots among the zeros of \(\Delta_{p,q}(t)\) is obtained by multiplying the previous limit by \(\varphi(r)\). This completes the proof.
\end{proof}

\subsection{Moments of roots of torus knots and their generating function}
\par For coprime integers \((p,q)=1\), let
\[
\mathcal{Z}_{p,q}=\{\zeta\in\mu_{pq}:\ \zeta^p\neq1,\ \zeta^q\neq1\}, \qquad
N_{p,q}=|\mathcal{Z}_{p,q}|=(p-1)(q-1),
\]
and define the \(m\)-th moment
\[
S_m(p,q)=\sum_{\zeta\in\mathcal{Z}_{p,q}}\zeta^m.
\]
Since the Alexander polynomial of the torus knot \(T_{p,q}\) is
\[
\Delta_{p,q}(t)=\frac{(t^{pq}-1)(t-1)}{(t^p-1)(t^q-1)},
\]
its roots are exactly the elements of \(\mathcal{Z}_{p,q}\), so the moments \(S_m(p,q)\) record the Fourier coefficients of the associated discrete measure on \(S^1\).

\par Our goal in this subsection is to obtain a closed formula for \(S_m(p,q)\) for all integers \(m\), and then study the generating function
\[
G_{p,q}(z):=\sum_{m\ge0} S_m(p,q)\,z^m,
\] whose meromorphic properties (location of poles and order of residues) encodes the oscillatory features of the sequence \((S_m(p,q))_m\).

\begin{proposition}
Let \(p,q\) be coprime positive integers and set
\[
\mathcal{Z}_{p,q}=\{\zeta\in\mu_{pq}:\ \zeta^p\neq1,\ \zeta^q\neq1\},
\qquad
N_{p,q}=|\mathcal{Z}_{p,q}|=(p-1)(q-1).
\]
For an integer \(m\) define the \(m\)-th moment
\[
S_m(p,q):=\sum_{\zeta\in\mathcal{Z}_{p,q}}\zeta^m.
\]
Then for every integer \(m\) one has the exact identity
\[
\; S_m(p,q)= \mathbf{1}_{pq|m}\;pq \;-\; \mathbf{1}_{p|m}\;p \;-\; \mathbf{1}_{q|m}\;q \;+\; 1 \; ,
\]
where \(\mathbf{1}_{d| m}\) denotes the indicator function of the divisibility \(d\mid m\).
Equivalently,
\[
S_m(p,q)=
\begin{cases}
(p-1)(q-1), & pq\mid m,\\[4pt]
1-p, & p\mid m,\ q\nmid m,\\[4pt]
1-q, & q\mid m,\ p\nmid m,\\[4pt]
1, & p\nmid m,\ q\nmid m.
\end{cases}
\]
\end{proposition}

\begin{proof}
We first record the elementary fact about complete sums of roots of unity: for any positive integer \(n\) and any integer \(m\),
\[
\sum_{\xi^{n}=1}\xi^{m}= \mathbf{1}_{n|m} n=
\begin{cases}
n,& n\mid m,\\[4pt]
0,& n\nmid m.
\end{cases}
\]
Set $n':=n/(n,m)$ and note that $n'=1$ precisely when $n|m$. The above relation holds because if \(n\mid m\) then each summand equals \(1\) and there are \(n\) summands; if \(n\nmid m\) the values \(\xi^m\) run through all \(n'\)-th roots of unity (each with multiplicity $(n, m)$), and hence have vanishing sum. Given an integer $k\geq 1$, let $\mu_k$ denote the $k$-th roots of unity in $\mathbb{C}^\times$. By inclusion–exclusion on the sets \(\mu_{pq},\mu_p,\mu_q\) we may write the desired restricted sum as a linear combination of complete sums:
\[
\begin{aligned}
S_m(p,q)
&=\sum_{\zeta\in\mu_{pq}\setminus(\mu_p\cup\mu_q)}\zeta^m
=\sum_{\zeta^{pq}=1}\zeta^m-\sum_{\zeta^p=1}\zeta^m-\sum_{\zeta^q=1}\zeta^m+\sum_{\zeta=1}\zeta^m\\[4pt]
&=\mathbf{1}_{pq|m}\,pq-\mathbf{1}_{p\mid m}\,p-\mathbf{1}_{q\mid m}\,q+1,
\end{aligned}
\]from which the result follows.
\end{proof}

\begin{theorem}\label{Theorem G p q residues}
Let \(p,q\) be coprime positive integers, set 
\[G_{p,q}(z):=\sum_{m\ge0} S_m(p,q)\,z^m
\]
\noindent for $|z|<1$. Then:
\begin{enumerate}
  \item \(G_{p,q}(z)\) admits the meromorphic closed form
  \[
  G_{p,q}(z)=\frac{pq}{1-z^{pq}}-\frac{p}{1-z^p}-\frac{q}{1-z^q}+\frac{1}{1-z},
  \qquad |z|<1,
  \]
  and all its singularities on the unit circle are simple poles located at the roots of unity \(\xi\) for which \(\xi^r=1\) for some \(r\in\{1,p,q,pq\}\).
  \item For a root of unity \(\xi\) with \(\xi^r=1\),
  \[
  \operatorname{Res}_{z=\xi}G_{p,q}(z)
  = -\xi^{-(pq-1)}+\xi^{-(p-1)}+\xi^{-(q-1)}-1.
  \]
\end{enumerate}
\end{theorem}

\begin{proof}

\par First we show that \begin{equation}\label{silly identity}
\sum_{\xi^n=1}\frac{1}{1-\xi z}=\frac{n}{1-z^n}.
\end{equation} Since $\prod_{\xi^n=1}(1-\xi z)=1-z^n$, taking logarithmic derivatives gives
\[
\frac{d}{dz}\log(1-z^n)=\sum_{\xi^n=1}\frac{d}{dz}\log(1-\xi z).
\]
Differentiating both sides, and multiplying by \(-z^{-(n-1)}\) yields the identity \eqref{silly identity}.
\par Since
\(\mathcal{Z}_{p,q}=\mu_{pq}\setminus(\mu_p\cup\mu_q)\), for \(|z|<1\) we find that
\[
G_{p,q}(z)=\sum_{\zeta\in\mathcal{Z}_{p,q}}\frac{1}{1-\zeta z}
=\frac{pq}{1-z^{pq}}-\frac{p}{1-z^p}-\frac{q}{1-z^q}+\frac{1}{1-z}.
\]When meromorphically continued to $\mathbb{C}$, this function has poles on the unit circle at roots of unity dividing \(1,p,q\) or \(pq\). Each denominator \(1-z^n\) has a simple zero at any \(\xi\), hence all singularities are simple poles and this proves part (1).
\par For (2), we compute residues. Let \(\xi\) be a root of unity; for each \(n\in\{1,p,q,pq\}\) with \(\xi^n=1\) the local expansion of the corresponding term \(c_n/(1-z^n)\) has residue
\[
\operatorname{Res}_{z=\xi}\frac{c_n}{1-z^n} = \frac{c_n}{-n\,\xi^{\,n-1}}.
\]
Summing these contributions with \(c_{pq}=pq,\ c_p=-p,\ c_q=-q,\ c_1=1\) yields the displayed residue formula. Writing the combination out and using \(c_n/n\in\{1,-1\}\) simplifies the expression to the compact form
\[
\operatorname{Res}_{x=\xi}G_{p,q}(z)
= -\xi^{-(pq-1)}+\xi^{-(p-1)}+\xi^{-(q-1)}-1,
\]
interpreting the right-hand side with only those summands for which the corresponding congruence \(\xi^n=1\) holds.
\end{proof}
The explicit closed formula
\[
S_m(p,q)=pq\,\mathbf{1}_{pq\mid m}-p\,\mathbf{1}_{p\mid m}-q\,\mathbf{1}_{q\mid m}+1
\]
shows that the sequence is strictly periodic of period \(pq\). Over one complete period, most values are equal to \(1\), with three distinguished classes of fluctuations at residue classes modulo \(p\), modulo \(q\), and modulo \(pq\). The moment generating function on the other hand has degree $pq$. A general theorem in the spirit of Frobenius and P\'olya implies that whenever a bounded sequence $(a_m)$ has a generating function whose only singularities on the unit circle are poles at roots of unity, that sequence must be a finite linear combination of exponential waves \[a_m:=\sum_{j=1}^t c_j e^{2\pi i\theta_j m}.\]  Thus the poles determine the oscillatory features of the sequence $(a_m)_m$.

\begin{proposition}
The sequence \(\{S_m(p,q)\}_{m\ge 0}\) is periodic of period \(pq\), and over one full period its mean is given by \[\frac{1}{pq}\sum_{m=0}^{pq-1} S_m(p,q)=0.\]
In particular, the moments oscillate symmetrically around \(0\). Moreover, the variance is given by
\[\frac{1}{pq}\sum_{m=0}^{pq-1}|S_m(p,q)|^2=(p-1)(q-1).\]
\end{proposition}
\begin{proof}
Partition the residue classes modulo \(pq\) into four types and record \(S_m(p,q)\) on each:
\[
\begin{array}{c|c|c}
\text{Congruence class of } m \bmod pq & \text{Number of classes} & S_m(p,q) \\ \hline
m \equiv 0 \pmod{pq} & 1 & (p-1)(q-1) \\[6pt]
m \equiv 0 \pmod p,\ \ m \not\equiv 0 \pmod q & q-1 & 1-p \\[6pt]
m \equiv 0 \pmod q,\ \ m \not\equiv 0 \pmod p & p-1 & 1-q \\[6pt]
\text{otherwise} & (p-1)(q-1) & 1
\end{array}
\]
\noindent Summing over a full period,
\[
\sum_{m=0}^{pq-1}S_m(p,q)
=(p-1)(q-1)+(q-1)(1-p)+(p-1)(1-q)+(p-1)(q-1)\cdot 1=0.
\]\noindent To compute the variance we use the same partition, but now sum the squares:
\[
\sum_{m=0}^{pq-1} |S_m(p,q)|^2
= 1\cdot\big[(p-1)(q-1)\big]^2 \;+\; (q-1)\cdot(1-p)^2 \;+\; (p-1)\cdot(1-q)^2 
\;+\; (p-1)(q-1)\cdot 1^2.
\]
Writing $a=p-1$ and $b=q-1$ for brevity, this becomes
\[
(a b)^2 \;+\; b a^2 \;+\; a b^2 \;+\; a b
= a b\,(a b + a + b + 1)
= a b \cdot pq,
\]
\noindent therefore
\[
\frac{1}{pq} \sum_{m=0}^{pq-1} |S_m(p,q)|^2 = ab = (p-1)(q-1),
\]
which proves that the variance is exactly $(p-1)(q-1)$.
\end{proof}

\par Next we establish a relationship between the residues on the circle circle and the variance. First we recall the well known Parseval's identity for fourier transforms of finite cyclic groups.
\begin{theorem}[Parseval's identity]\label{parseval thm}
Let \(N\ge1\) and let \((a_m)_{m\in\mathbb Z}\) be an \(N\)-periodic complex sequence (so \(a_{m+N}=a_m\) for all \(m\)) and let $\omega$ be an \(N\)-th root of unity. We set
\[
\widehat a(\omega)\;:=\;\sum_{m=0}^{N-1} a_m\,\omega^m.
\]
Then, one has the following identity:
\[
\sum_{m=0}^{N-1}\big|a_m\big|^2 \;=\; \frac{1}{N}\sum_{\omega^N=1}\big|\widehat a(\omega)\big|^2.
\]
\end{theorem}

\begin{corollary}
Set \(a_m:=S_m(p,q)\) and recall that $G_{p,q}(z)=\sum_{m\ge 0} a_m z^m$. For each \(\omega\in \mu_{pq}\), let us denote $R(\omega):=\operatorname{Res}_{z=\omega}G_{p,q}(z)$. Then the finite Fourier transform \(\widehat a(\omega)\) and the residue are related by
\[
\widehat a(\omega) = -pq\omega^{\,pq-1} R(\omega),
\]
and consequently the variance equals the sum of squared moduli of the residues:
\[
\frac{1}{pq}\sum_{m=0}^{pq-1}\big|a_m\big|^2
\;=\; \sum_{\omega^{pq}=1} \big|R(\omega)\big|^2.
\]
\end{corollary}

\begin{proof}
Set $N:=pq$. Because \(a_m\) is \(N\)-periodic we have the finite geometric identity (valid for \(|z|<1\))
\[
G_{p,q}(z)=\sum_{m\ge0}a_m z^m
=\frac{1}{1-z^N}\sum_{m=0}^{N-1} a_m z^m.
\]
Multiply both sides by \(1-z^N\) and evaluate the principal part at \(z=\omega\). Near \(z=\omega\) we have the expansion
\[
1-z^N = -N\omega^{\,N-1}(z-\omega)+O((z-\omega)^2),
\]
so the principal part of \(G_{p,q}(z)\) at \(z=\omega\) is
\[
\frac{\sum_{m=0}^{N-1} a_m \omega^m}{1-z^N}
=\frac{\widehat a(\omega)}{-N\omega^{\,N-1}(z-\omega)} + \text{(holomorphic)}.
\]
Comparing with the definition of the residue yields
\[
R(\omega)=\operatorname{Res}_{z=\omega}G_{p,q}(z) = -\frac{\widehat a(\omega)}{N\omega^{\,N-1}},
\]
hence \(\widehat a(\omega)=-N\omega^{\,N-1}R(\omega)\). From Theorem \ref{parseval thm}, we obtain the formula:
\[
\frac{1}{N}\sum_{m=0}^{N-1}\big|a_m\big|^2
\;=\; \sum_{\omega^N=1} \big|R(\omega)\big|^2.\]
\end{proof}
\noindent As the number of residues on the unit circle increases, so does the variance, which grows linearly as $p,q\rightarrow \infty$.

\section{Mahler measure and homology growth}
In this section, we discuss the growth of homology in cyclic covers of $S^3$, which are branched along a link. First we set up some basic notation. 
\subsection{Asymptotic growth in torsion homology}
\par Let $\cL\subset S^3$ be a link with component knots $\cK_1, \dots, \cK_r$. Let $X_{\cL}:=S^3\setminus \cL$ and denote by $\pi_{\cL}^{\op{ab}}:\widetilde{X_{\cL}}^{\op{ab}}\rightarrow X_{\cL}$ the maximal abelian cover of $X_{\cL}$. We note that by maximality this is a Galois cover and its Galois group $\op{Gal}(\widetilde{X_{\cL}}^{\op{ab}}/X_{\cL})$ is isomorphic to $\Z^r$ via a natural isomorphism which is the composite:
\[\Phi: \op{Gal}(\widetilde{X_{\cL}}^{\op{ab}}/X_{\cL})\xrightarrow{\sim} \pi_1(X_{\cL})^{\op{ab}}\xrightarrow{\sim} \Z^r\] with the latter map taking the $i$-th meridian to the vector $e_i$ with $1$ in the $i$-th spot and $0$ in other positions. In order to specify a $\Z$-cover, we specify a surjective homomorphism 
\[\op{Gal}(\widetilde{X_{\cL}}^{\op{ab}}/X_{\cL})\twoheadrightarrow \Z\]
factoring $\Phi$. Given an admissible vector $z$, define 
\[\Phi_z: \op{Gal}(\widetilde{X_{\cL}}^{\op{ab}}/X_{\cL})\rightarrow \Z\] by $\Phi_z:=\alpha_z\circ \Phi$, where $\alpha_z(t_1, \dots, t_r):=\sum_i z_i t_i$. The condition that $z$ is admissible translates to the map $\Phi_z$ being surjective. The choice of $\Phi_z$ coincides with a choice of $\Z$-cover 
\[\pi_{\cL, z}: X_{\cL, z}\rightarrow X_{\cL}.\]Letting $\mathbf{1}=(1,\dots, 1)$, we call $X_{\cL,\mathbf{1}}$ the \emph{total linking number} covering space of $X_{\cL}$. For ease of notation we abbreviate $X_z$ to $X_z$, and $\pi_z$ to $\pi_z$.

\par Given an integer $m\geq 0$, let $\pi_z^m: X_{z,m}\rightarrow X$ be the unique $\Z/m \Z$-subcover of $\pi_z$ and
\begin{equation}\label{M_{z,m} defn}
\tilde{\pi}_{z}^m : M_{z,m} \longrightarrow S^3
\end{equation}
denote the \emph{Fox completion} of the covering map \(\pi_{z}^m\); see \cite[§10.2]{morishita2011knots} for a detailed account of its construction and properties. By definition, \(\tilde{\pi}_{z}^m\) is a branched covering of the 3-sphere whose branch locus coincides with the link \(\cL\).

\par In the special case in which $\cL=\cK$ is a knot, there is a unique $\Z$-cover and we shall simply denote by $\tilde{\pi}^m: M_m\rightarrow S^3$ the unique $\Z/m\Z$-cover. Consider the sequence of numbers $(h_m)$ where $h_m(\cK):=|H_1(M_m)|$, where one sets $h_m(\cK):=0$ if $H_1(M_m)$ is infinite. 

\par Given polynomials
\[
f(x) = a_m x^m + a_{m-1}x^{m-1} + \cdots + a_0, \qquad 
g(x) = b_n x^n + b_{n-1}x^{n-1} + \cdots + b_0
\]
in \(\mathbb{C}[x]\), the \emph{resultant} of \(f\) and \(g\), denoted \(\operatorname{Res}(f,g)\), is defined as the determinant of the \((m+n)\times(m+n)\) \emph{Sylvester matrix}
\[
S(f,g) =
\begin{pmatrix}
a_m & a_{m-1} & \cdots & a_0 & 0 & \cdots & 0 \\
0 & a_m & a_{m-1} & \cdots & a_0 & \cdots & 0 \\
\vdots &  &  &  &  &  & \vdots \\
0 & \cdots & 0 & a_m & a_{m-1} & \cdots & a_0 \\[4pt]
b_n & b_{n-1} & \cdots & b_0 & 0 & \cdots & 0 \\
0 & b_n & b_{n-1} & \cdots & b_0 & \cdots & 0 \\
\vdots &  &  &  &  &  & \vdots \\
0 & \cdots & 0 & b_n & b_{n-1} & \cdots & b_0
\end{pmatrix}.
\]
\noindent If we factor \(f\) and \(g\)
\[
f(x)=a_m\prod_{i=1}^m (x-\alpha_i), \qquad
g(x)=b_n\prod_{j=1}^n (x-\beta_j),
\]
then the resultant can be expressed multiplicatively as
\begin{equation}\label{basicresultantproperty}\operatorname{Res}(f,g)
= a_m^n b_n^m \prod_{i=1}^m \prod_{j=1}^n (\alpha_i - \beta_j)=a_m^n \prod_{i=1}^m g(\alpha_i)
= (-1)^{mn} b_n^m \prod_{j=1}^n f(\beta_j).
\end{equation}

\begin{theorem}[Fox,Weber]
    Let $\cK$ be a knot and $h_m(\cK)$ be as above, then 
    \[h_m(\cK)=|\op{Res}\left(t^m-1, \Delta_{\cK}(t)\right)|.\]
\end{theorem}
\begin{proof}
    We refer to \cite{Weber} for the proof.
\end{proof}
\begin{remark}
    The above formula was first established by Fox under the additional assumption that $H_1(\widetilde{X_{\cK}}^{\op{ab}})$ is a direct sum of cyclic modules over the Laurent series ring. This assumption was subsequently relaxed by Weber.
\end{remark}
\noindent Write $\Delta_{\cK}(t)=b_n\prod_{j=1}^n (t-\beta_j)$, then in view of \eqref{basicresultantproperty}, one has that 
\[h_m(\cK)=|b_n^m\prod_{j=1}^m (\beta_j^m-1)|.\]
\noindent Gordon \cite{Gordon} showed that $(h_n(\cK))_n$ is a periodic sequence if and only if all roots of $\Delta_{\cK}(t)$ are roots of unity. This is indeed the case when $\cK$ is a torus knot.
\par More generally, let $\op{M}(\Delta_\cK)$ denote the \emph{Mahler measure} of $\Delta_{\cK}$ defined as follows:
\[\op{M}(\Delta_\cK):=|b_n|\prod_j \op{max}\left\{1, |\beta_j|\right\}.\]
\begin{theorem}[Acuna--Short]
    With respect to notation above, 
    \[\lim_{\substack{n\rightarrow \infty,\\
    h_n(K)\neq 0}}h_n(K)^{1/n}=\op{M}(\Delta_\cK).\]
\end{theorem}
\begin{proof}
    The result is \cite[Theorem 1]{acunashort}.
\end{proof}

Next we recall the multivariable Mahler measure and state the fundamental theorem of Silver and Williams \cite{SW1, SW2} that connects this analytic invariant to growth of torsion in homology of finite abelian covers.

\begin{definition}
Let $P\in\mathbb{C}[t_1^{\pm1},\dots,t_r^{\pm1}]$ be a Laurent polynomial in $r$ variables. The (logarithmic) \emph{Mahler measure} of $P$ is
\[
m(P):=\int_{0}^1\cdots\int_0^1\log\big|P(e^{2\pi i\theta_1},\dots,e^{2\pi i\theta_r})\big|\,d\theta_1\cdots d\theta_r,
\]
and the multiplicative Mahler measure is $M(P):=e^{m(P)}$.
\end{definition}
Let $\cL\subset S^3$ be an $r$–component link with $\Delta_{\cL}\not\equiv 0$. For each $n\ge1$, let $X_{\cL,n}$ be the finite abelian cover of $X_{\cL}$ with Galois group $(\Z/n\Z)^r$, and let $M_n$ denote its Fox completion. Set
\[
\tau_n := \#H_1(M_n;\Z)_{\mathrm{tors}}.
\]
More generally, for a finite index subgroup $G\subseteq \Z^r$, let $X_{\cL,G}$ be the corresponding cover with Galois group $\Z^r/G$, and let $M_G$ be its Fox completion. Define
\[
\tau_G := \#H_1(M_G;\Z)_{\mathrm{tors}}.
\]

\begin{theorem}[Silver--Williams]\label{thm:SW}
 With respect to notation above, one has that
\[
\lim_{|G|\to\infty}\; \frac{1}{[\Z^r:G]}\log\tau_G \;=\; m\big(\Delta_{\cL}\big).
\]
\end{theorem}
\begin{proof}
    This result is \cite[Theorem 2.1]{SW1}.
\end{proof}
\noindent In particular, the above result implies that 
\[\lim_{n\to\infty}\; \frac{1}{n^r}\log\tau_n \;=\; m\big(\Delta_{\cL}\big).\]

\subsection{Consequences and explicit computation for torus links}

We now apply Theorem~\ref{thm:SW} to the family of torus links $T_{p,q}$. The key simple observation we exploit is the following lemma, which allows immediate evaluation of the Mahler measure of any polynomial that is a product or quotient of binomials of the form $\mathbf{t}^\alpha-1$, where $\mathbf{t}^\alpha:=t_1^{\alpha_1}\cdots t_r^{\alpha_r}$.

\begin{lemma}\label{lem:binomial-mahler-zero}
Let $\alpha=(\alpha_1,\dots,\alpha_r)\in\Z^r$ be a nonzero integer vector and set
\[
P_\alpha(t_1,\dots,t_r):=\mathbf{t}^\alpha-1,
\]
then, $m(P_\alpha)=0$.
\end{lemma}

\begin{proof}
We write $t_j = e^{2\pi i \theta_j}$ for $\theta_j \in [0,1)$, so that
\[
\mathbf{t}^\alpha = t_1^{\alpha_1}\cdots t_r^{\alpha_r}
= e^{2\pi i (\alpha_1 \theta_1 + \cdots + \alpha_r \theta_r)}
= e^{2\pi i \ell(\theta)},
\]
where $\ell(\theta) = \sum_{j=1}^r \alpha_j \theta_j$ is a linear form on $\mathbb{R}^r$.
Hence
\[
m(P_\alpha)
= \int_{\mathbb{T}^r} \log|\mathbf{t}^\alpha - 1|\, d\theta_1\cdots d\theta_r
= \int_{[0,1)^r} \log|e^{2\pi i \ell(\theta)} - 1|\, d\theta_1\cdots d\theta_r.
\]
Since the integrand depends only on $\ell(\theta)$, we may perform an orthogonal change of coordinates
that separates the direction of $\alpha$ from the directions orthogonal to it.
Let $v_1,\dots,v_r$ be an orthonormal basis of $\mathbb{R}^r$ such that
\[
v_1 = \frac{\alpha}{\|\alpha\|} =
\frac{1}{\sqrt{\sum_j \alpha_j^2}} (\alpha_1,\dots,\alpha_r),
\]
and let $V=[v_1\,v_2\,\cdots\,v_r]\in O(r)$ be the orthogonal matrix whose $i$-th column is $v_i$.
Define new coordinates $(\phi,\psi_2,\dots,\psi_r)$ by
\[
\begin{pmatrix}
\phi \\ \psi_2 \\ \vdots \\ \psi_r
\end{pmatrix}
= V^{-1}
\begin{pmatrix}
\theta_1 \\ \theta_2 \\ \vdots \\ \theta_r
\end{pmatrix}.
\]
Because $V$ is orthogonal, the Jacobian of this transformation satisfies
$d\theta_1\cdots d\theta_r = d\phi\,d\psi_2\cdots d\psi_r$.
In these coordinates one has
\[
\ell(\theta) = \alpha \cdot \theta
= \alpha^{\!T} V
\begin{pmatrix}
\phi \\ \psi_2 \\ \vdots \\ \psi_r
\end{pmatrix}
= (\|\alpha\|, 0, \dots, 0)
\begin{pmatrix}
\phi \\ \psi_2 \\ \vdots \\ \psi_r
\end{pmatrix}
= \|\alpha\| \phi.
\]
Thus $\ell(\theta)$ depends only on the first coordinate $\phi$,
and the integrand $\log|\mathbf{t}^\alpha - 1|$ becomes
$\log|e^{2\pi i \|\alpha\| \phi} - 1|$.
The integral over the orthogonal variables $\psi_2,\dots,\psi_r$
merely contributes a factor of $1$ because the integrand is constant in those directions,
so we obtain
\[
\int_{\mathbb{T}^r} \log|\mathbf{t}^\alpha - 1|\, d\theta_1\cdots d\theta_r
= \int_0^1 \log|e^{2\pi i \|\alpha\| \phi} - 1|\, d\phi.
\]
Since the function $\phi \mapsto \log|e^{2\pi i \|\alpha\| \phi} - 1|$ is $1/\|\alpha\|$–periodic,
we may rescale $\phi \mapsto \phi / \|\alpha\|$ without changing the integral, yielding
\[
\int_0^1 \log|e^{2\pi i \phi} - 1|\, d\phi.
\]
This one–dimensional integral vanishes, as it is well known that the logarithmic Mahler measure of the polynomial
$z-1$ equals zero and hence,
$m(P_\alpha)=0$.
\end{proof}

\begin{corollary}\label{cor:torus-measure-zero}
Let $T_{p,q}$ be a torus link (with $d=\gcd(p,q)$ components) and let
\[
\Delta_{T_{p,q}}(t_1,\dots,t_d)\in\Z[t_1^{\pm1},\dots,t_d^{\pm1}]
\]
be its (normalized) multivariable Alexander polynomial. Then
\[
m\big(\Delta_{p,q}\big)=0.
\]
\end{corollary}

\begin{proof}
The multivariable Alexander polynomial of a torus link \eqref{MAP eqn} factors as a product and quotient of binomials of the form $\mathbf{t}^\alpha-1$, with integer exponent vectors $\alpha$. Since Mahler measure is additive under multiplication, Lemma~\ref{lem:binomial-mahler-zero} implies that $m(\Delta_{p,q})=0$.
\end{proof}
Combining Theorem~\ref{thm:SW} and Corollary~\ref{cor:torus-measure-zero} yields the following immediate consequence for torsion growth in the abelian towers of torus-link complements.

\begin{corollary}\label{cor 4.8}
Let $\cL=T_{p,q}$ be any torus link and let $M_n$ be the Fox-completed $(\Z/n\Z)^d$–cover of $S^3$ branched over~$\cL$ (notation as in Theorem~\ref{thm:SW}). Then
\[
\lim_{n\to\infty}\frac{1}{n^d}\log\#\big(H_1(M_n;\Z)_{\mathrm{tors}}\big)
= m\big(\Delta_{p,q}\big) \;=\; 0.
\]
\end{corollary}

\begin{proof}
This is an immediate conjunction of Theorem~\ref{thm:SW} and Corollary~\ref{cor:torus-measure-zero}.
\end{proof}

\section{Iwasawa invariants}
\subsection{Classical Iwasawa theory}
There is a rich and surprisingly precise dictionary between three-dimensional topology (knots and $3$--manifolds) and the arithmetic of number fields; this perspective, often called \emph{arithmetic topology}, was initiated by Artin and Mazur \cite{ArtinMazur}. For a comprehensive overview of this analogy, we refer to \cite{morishita2011knots}. The dictionary begins with the observation that the arithmetic scheme $\mathfrak{X}=\Spec\cO_K$ of the ring of integers of a number field $K$ behaves, from the point of view of \'etale cohomology and homotopy, like a closed oriented $3$--manifold. For instance, the Tate-modified \'etale cohomology $\hat{H}^i(\mathfrak{X},-)$ has cohomological dimension $3$, and there is a canonical \emph{fundamental class}:
\[
\hat{H}^3(\mathfrak{X},\mathbb{G}_{m,\mathfrak X}) \;\xrightarrow{\;\sim\;}\;\; \Q/\Z
\]
which plays the role of the orientation class of a $3$--manifold. The \emph{Artin--Verdier duality theorem} asserts that for any constructible \'etale sheaf $\mathfrak M$ on $\mathfrak X$ with dual $\mathfrak M^\vee$, there is a perfect pairing
\[
\hat{H}^i(\mathfrak{X}, \mathfrak{M}^\vee)\times \op{Ext}_{\mathfrak{X}}^{3-i}(\mathfrak{M}, \mathbb{G}_{m, \mathfrak{X}})\;\longrightarrow\; \hat{H}^3(\mathfrak{X}, \mathbb{G}_{m, \mathfrak{X}})\;\xrightarrow{\sim}\;\Q/\Z.
\]
Here, $\op{Ext}_{\mathfrak{X}}^{3-i}(\mathfrak{M}, \mathbb{G}_{m,\mathfrak X})$ plays the role of homology, obtained in the arithmetic setting as a derived functor into the dualizing sheaf $\mathbb G_m$. On the topological side of the dictionary, if $M$ is a closed oriented $3$--manifold, the Poincar\'e duality theorem asserts that cap product with a fundamental class yields a perfect pairing
\[
H^i(M)\;\times\; H_{3-i}(M)\;\longrightarrow\; H_0(M)\cong \Z,
\]
so that cohomology and homology are naturally dual in complementary degrees.
\par This dictionary is quite refined and extends well beyond global duality results. A prime ideal $\mathfrak p\subset \cO_K$ corresponds to a knot in a $3$--manifold: its local field $K_{\mathfrak p}$ or local scheme $\Spec \cO_{K,\mathfrak p}$ corresponds to a tubular neighbourhood or boundary torus, and the inertia and decomposition subgroups correspond to the subgroups generated by the meridian and longitude of the knot complement. Abelian coverings of $M$ are classified by $H_1(M,\Z)$, which is identified with the abelianization of the fundamental group of $M$. On the arithmetic side of the correspondence, the class group $\op{Cl}(K)$ of $K$ is by definition, the Picard group of $\cO_K$. Fix an algebraic closure $\overline{K}/K$ and let $K^{nr}\subset \overline{K}$ be the maximal unramified extension of $K$. Then, $\op{Gal}(K^{nr}/K)$ is an analogue of the fundamental group of a $3$-manifold, and by class field theory, its abelianization is isomorphic to $\op{Cl}(K)$. Thus, the class group of $K$ is a natural analogue of $H_1(M)$. Class field theory gives an explicit description of all abelian extensions of $K$ that are contained in $\overline{K}$. There are topological analogues of class field theory, cf. \cite{Nibo, NiboUeki}.
\par Let $\ell$ be a prime number. According to this analogy described above, the Alexander polynomial of a $\Z$-cover can be $\ell$-adically completed to give analogues of Iwasawa polynomials associated to $\Z_\ell$-extensions of number fields. In order to explain this, let us discuss the number theoretic analogue first. Let $\Z_\ell:=\varprojlim_n \Z/\ell^n \Z$ be the ring of $\ell$-adic integers. A $\Z_\ell$-extension $K_\infty/K$ is an infinite Galois extension with Galois group $\Gamma:=\op{Gal}(K_\infty/K)$ isomorphic to $\Z_\ell$. For each integer $n\geq 0$, there is a unique extension $K_n/K$ contained in $K_\infty$ such that $\op{Gal}(K_n/K)\simeq \Z/\ell^n\Z$. This gives a tower of number fields
\[K=K_0\subset K_1\subset K_2\subset \dots \subset K_n\subset K_{n+1}\subset \dots \subset K_\infty.\]
\noindent Denote by $\op{Cl}(K_n)$ the class group of $K_n$ and write $\# \op{Cl}(K_n)=\ell^{e_n} h_n'$ where $h_n'$ is coprime to $\ell$.
\begin{theorem}[Iwasawa]
    There exists $n_0\geq 0$ and invariants $\mu, \lambda\in \Z_{\geq 0}$ and $\nu\in \Z$ such that for all integers $n\geq 0$, 
    \[e_n=p^n \mu+n \lambda+\nu.\]
\end{theorem}
\noindent The Iwasawa invariants are associated to a polynomial that is naturally associated to the module $\mathcal{X}:=\varprojlim_{n} \op{Cl}(K_n)[p^\infty]$, where the inverse limit is taken with respect to norm maps. Let $\Gamma:=\op{Gal}(K_\infty/K)$ and identify $\Gamma/\Gamma^{p^n}$ with the Galois group $\op{Gal}(K_n/K)$. Then the Iwasawa algebra is defined as the completed group algebra 
\[\Z_p\llbracket \Gamma\rrbracket:=\varprojlim_n \Z_p[\Gamma/\Gamma^{p^n}].\] 

\par Let $K_{\op{cyc}}/K$ be the unique $\Z_\ell$-extension of $K$ contained in the infinite cyclotomic extension $K(\mu_{\ell^\infty})$ generated over $K$ by the $\ell$-power roots of unity. Then, it is conjectured the the $\mu$-invariant vanishes for $K_{\op{cyc}}/K$. This has been proven by Fererro and Washington \cite{ferrerowash} in the special case when $K/\Q$ is an abelian number field. 
\subsection{Iwasawa theory of abelian branched covers of spheres}
Let $\cL\subset S^3$ be an $r$--component link and let
\[
\Lambda_r=\mathbb{Z}[X_1^{\pm1},\dots,X_r^{\pm1}]
\qquad\text{and}\qquad
\Lambda=\mathbb{Z}[X^{\pm1}]
\]
be the Laurent polynomial rings identified with the group rings of
$\operatorname{Gal}(\widetilde{X_\cL}^{\operatorname{ab}}/X_\cL)$ and
$\operatorname{Gal}(X_z/X_\cL)\simeq\mathbb{Z}$ respectively via the chosen
isomorphisms $\Phi$ and $\Phi_z$. Suppose that $r\geq 2$. Then if
\(\Delta_{\cL}(X_1,\dots,X_r)\in\Lambda_r\) is the multivariable Alexander
polynomial, then for an admissible integral vector
$z=(z_1,\dots,z_r)\in\mathbb{Z}^r$ we write the specialization
\[
\Delta_z(X):=(X-1)\Delta_{\cL}(X^{z_1},\dots,X^{z_r})\in\Lambda.
\]

Fix a prime $\ell$.  Let $\widehat{\Lambda}=\mathbb{Z}_\ell\llbracket T\rrbracket$
be the $\ell$--adic completion (the one-variable Iwasawa algebra).  We embed
$\Lambda$ into $\widehat{\Lambda}$ by sending $X\mapsto 1+T$.  (Under this
map $X^{-1}$ is sent to $(1+T)^{-1}=1-T+T^2-\cdots$ inside
$\mathbb{Z}_\ell\llbracket T\rrbracket$.)  The completed (or \emph{$\ell$--adic})
$z$--Alexander polynomial is therefore
\[
\widehat{\Delta}_z(T):=\Delta_z(1+T)\in\mathbb{Z}_\ell\llbracket T\rrbracket .
\]

By the Weierstrass preparation theorem there is a unique factorization
\[
\widehat{\Delta}_z(T)=\ell^{\mu_z}\,P_z(T)\,U_z(T)
\]
where $\mu_z\in\mathbb{Z}_{\ge0}$, $P_z(T)$ is a distinguished polynomial
(i.e. monic and all nonleading coefficients divisible by $\ell$) and
$U_z(T)\in\mathbb{Z}_\ell\llbracket T\rrbracket^\times$ is a unit.  The
\emph{$\lambda$--invariant} $\lambda_z$ is $\deg P_z$, equivalently the
smallest integer $\lambda\ge0$ with
\(\ell^{-\mu_z}\widehat{\Delta}_z(T)\equiv T^{\lambda}\pmod{\ell}\) in
$\mathbb{F}_\ell\llbracket T\rrbracket$.  The integers $\mu_z,\lambda_z$ are
the usual Iwasawa invariants attached to the $\mathbb{Z}_\ell$--tower
determined by $z$; the third invariant $\nu_z$ is defined by the classical
Iwasawa asymptotic for the orders of the homology groups in the tower (see
the theorem below).

For any finite abelian group $A$ write $|A|$ for its cardinality, and set
$|A|=0$ if $A$ is infinite.  For $p=\ell$ we write $v_\ell(\cdot)$ for the
$\ell$--adic valuation normalized so $v_\ell(\ell)=1$.

\par The following is the standard topological analogue of the Iwasawa growth
theorem.

\begin{theorem}[Mayberry--Murasugi]\label{MMthm}
Let $z$ be an admissible integral vector and suppose $\widehat{\Delta}_z\neq0$.
Write $v=v(z):=\max_i v_\ell(z_i)$. For $n\geq 0$, let $M_{z,\ell^n}$ be the branched cover of $S^3$ given by \eqref{M_{z,m} defn}. Then
\begin{enumerate}
  \item For every integer $n\ge v$,
  \[
  \big|H_1(M_{z,\ell^n};\mathbb{Z})\big|
  =\big|H_1(M_{z,\ell^v};\mathbb{Z})\big|
  \cdot
  \prod_{\substack{\zeta^{\ell^n}=1\\\zeta^{\ell^v}\ne1}}
  \big|\Delta_z(\zeta)\big|.
  \]
  In particular, the growth of the finite part of the homology groups in
  the tower \((M_{z,\ell^n})_{n\ge v}\) is governed by the values of
  \(\Delta_z\) at $\ell$--power roots of unity.

  \item If $|H_1(M_{z,\ell^n};\mathbb{Z})|<\infty$ for all sufficiently
  large $n$, then there exist integers $\mu_z,\lambda_z,\nu_z$ (the
  Iwasawa invariants attached to the tower determined by $z$) and an integer $n_0$
  such that for all $n\ge n_0$,
  \[
  v_\ell\big(|H_1(M_{z,\ell^n};\mathbb{Z})|\big)
  =\ell^n\mu_z + n\lambda_z + \nu_z.
  \]
  Equivalently, one has that $|H_1(M_{z,\ell^n};\mathbb{Z})|=\ell^{\ell^n\mu_z+n\lambda_z+\nu_z}$ for $n\geq n_0$.
\end{enumerate}
\end{theorem}

When $\cL$ is a knot ($r=1$) then an admissible vector $z$ is $\pm1$ and the
single-variable Alexander polynomial satisfies $\Delta_{\cL}(X)=\Delta_{\cL}(X^{-1})$. We denote the Iwasawa invariants by $\mu_{\cL}, \lambda_{\cL}$ and $\nu_{\cL}$ (which do not depend on $z$). One has that $\Delta_{\cL}(1)=\pm1$ and hence $\widehat{\Delta}_{\cL}(T)=\Delta_{\cL}(1+T)$
is a unit in $\mathbb{Z}_\ell\llbracket T\rrbracket$. Therefore for any knot, $\mu_\cL=\lambda_\cL=0$. It remains to determine the $\nu$--invariant.  We show that for torus knots
$\nu_{\cL}=0$ as well (under the hypotheses considered below). For ease of notation, set $\mu(p,q):=\mu_{T_{p,q}}$, $\lambda(p,q):=\lambda_{T_{p,q}}$ and $\nu(p,q):=\nu_{T_{p,q}}$. Let $v_\ell$ denote the valuation normalized by $v_\ell(\ell)=1$.

\begin{theorem}\label{thm 5.3}
Let $T_{p,q}$ be the torus knot with $\gcd(p,q)=1$, and fix a prime $\ell$. Then, setting $r:=v_\ell(pq)$, one has that \[|H_1(M_{p^n})|=q^{\ell^{\op{min}(n,r)}-1}.\]
Thus,\[\mu(p,q)=\lambda(p,q)=\nu(p,q)=0.\]
\end{theorem}

\begin{proof}
By part~(1) of Theorem \ref{MMthm}, the order of the first homology group of the
$\ell^n$--fold cyclic branched cover $M_{\ell^n}$ is given by
\[
\big|H_1(M_{\ell^n};\mathbb{Z})\big|
=\prod_{\substack{\zeta^{\ell^n}=1\\ \zeta\ne1}}
\big|\Delta_{p,q}(\zeta)\big|
=\prod_{k=1}^{n}
\operatorname{Norm}_{\mathbb{Q}(\zeta_{\ell^k})/\mathbb{Q}}
\big(\Delta_{p,q}(\zeta_{\ell^k})\big),
\]
where $\zeta_{\ell^k}$ denotes a primitive $\ell^k$--th root of unity. Indeed, grouping together the factors corresponding to primitive
$\ell^k$--th roots converts the product over all $\ell^n$--th roots of unity
into a product of norms.

\vspace{0.5em}
\noindent
\emph{Case 1.} Suppose first that $\ell\nmid pq$.
Then each of the four norm factors of the form
$\operatorname{Norm}(1-\zeta_{\ell^k}^m)$
appearing in the numerator and denominator of
$\Delta_{p,q}(\zeta_{\ell^k})$
is equal, since $m$ is invertible modulo $\ell^k$.
Consequently, the numerator and denominator cancel term-by-term, and the
entire product equals~$1$.
Hence $|H_1(M_{\ell^n};\mathbb{Z})|=1$ for all~$n$.
In particular, all three Iwasawa invariants vanish:
$\mu=\lambda=\nu=0$.

\vspace{0.5em}
\noindent
\emph{Case 2.} Next assume that $\ell$ divides $p$ and write $p=\ell^r p'$ where $\ell\nmid p'$. A similar argument applies when $\ell|q$. We find that \[\begin{split}&\prod_{k=1}^{n}
\operatorname{Norm}_{\mathbb{Q}(\zeta_{\ell^k})/\mathbb{Q}}
\big(\Delta_{p,q}(\zeta_{\ell^k})\big)\\
=& \prod_{k=1}^n \left(\op{Norm}_{\Q(\xi_{\ell^k})/\Q}\left(\frac{(X^{pq}-1)}{(X^p-1)}_{|X=\xi_{\ell^k}}\right)\frac{\op{Norm}_{\Q(\xi_{\ell^k})/\Q}(\xi_{\ell^k}-1)}{\op{Norm}_{\Q(\xi_{\ell^k})/\Q}(\xi_{\ell^k}^q-1)}\right)\end{split}\] On the other hand, since $q$ is coprime to $p$, $\ell\nmid q$, and thus, \[\op{Norm}_{\Q(\xi_{\ell^k})/\Q}(\xi_{\ell^k}^q-1)=\op{Norm}_{\Q(\xi_{\ell^k})/\Q}(\xi_{\ell^k}-1).\]Note that \[\frac{(X^{pq}-1)}{(X^{p}-1)}=\prod_{\substack{\zeta^q=1\\ \zeta\neq 1} } (X^p-\zeta).\] For $k\leq r$, $\xi_{\ell^k}^p=\xi_{\ell^k}^{\ell^r p'}=1$ and therefore
\[\frac{(X^{pq}-1)}{(X^{p}-1)}_{|X=\xi_{\ell^k}}=\prod_{\substack{\zeta^q=1\\ \zeta\neq 1} } (1-\zeta)=\frac{Y^q-1}{Y-1}_{|Y=1}=\left(\sum_{j=0}^{q-1} Y^j\right)_{|Y=1}=q.\]
Consequently for $k\leq r$ we find that \[\op{Norm}_{\Q(\xi_{\ell^k})/\Q}\left(\frac{(X^{pq}-1)}{(X^p-1)}_{|X=\xi_{\ell^k}}\right)=q^{[\Q(\xi_{\ell^k}):\Q]}=q^{\varphi(\ell^k)}=q^{\ell^k-\ell^{k-1}}.\]
\noindent Next consider the case when $k>r$, then $\xi_{\ell^k}^p=\xi_{\ell^{k-r}}^{p'}\neq 1$. Consequently one has that
\[\op{Norm}_{\Q(\xi_{\ell^k})/\Q}\left(\frac{(X^{pq}-1)}{(X^p-1)}_{|X=\xi_{\ell^k}}\right)=\frac{\op{Norm}_{\Q(\xi_{\ell^k})/\Q} (\xi_{\ell^{k-r}}^{p'q}-1)}{\op{Norm}_{\Q(\xi_{\ell^k})/\Q}(\xi_{\ell^{k-r}}^{p'}-1)}=1.\]
Putting it all together, we find that \[\begin{split}|H_1(M_{p^n})|=&\prod_{k=1}^{n}
\operatorname{Norm}_{\mathbb{Q}(\zeta_{\ell^k})/\mathbb{Q}}
\big(\Delta_{p,q}(\zeta_{\ell^k})\big)\\=&\prod_{k=1}^{\op{min}(n,r)}
\operatorname{Norm}_{\mathbb{Q}(\zeta_{\ell^k})/\mathbb{Q}}
\big(\Delta_{p,q}(\zeta_{\ell^k})\big)\\
=&\prod_{k=1}^{\op{min}(n,r)}q^{\ell^k-\ell^{k-1}}\\= &q^{\ell^{\op{min}(n,r)}-1}.\end{split}\] This completes the proof.
\end{proof}

\subsection{Torus \emph{links} and the Hosokawa polynomial}
Now suppose $T_{p,q}$ is the torus \emph{link} with $d=\gcd(p,q)\ge2$
components.  Write $p=dp'$ and $q=dq'$ with $\gcd(p',q')=1$.  For the
$d$--variable Alexander polynomial (where one uses variables
$X_1,\dots,X_d$ corresponding to the $d$ components) one has the standard
formula \eqref{MAP eqn}:
\[
\Delta_{p,q}(X_1,\dots,X_d)
= \frac{\big((X_1\cdots X_d)^{p'q'}-1\big)^d}
{\big((X_1\cdots X_d)^{p'}-1\big)\,\big((X_1\cdots X_d)^{q'}-1\big)}.
\]
For an admissible specialization $z=(z_1,\dots,z_d)$ set $\alpha=\sum_i z_i$.
The single-variable specialization is
\[
\Delta_{p,q}^z(X)
= \frac{\big(X^{\alpha p'q'}-1\big)^d\,(X-1)}
{\big(X^{\alpha p'}-1\big)\,\big(X^{\alpha q'}-1\big)}.
\]
The factor $(X-1)^{d-1}$ divides $\Delta_{p,q}^z(X)$, and the
quotient (the Hosokawa polynomial) is
\[
L_{p,q}^z(X)=\frac{g_{\alpha p'q'}(X)^d}{g_{\alpha p'}(X)\,g_{\alpha q'}(X)},
\qquad\text{where } g_k(X)=\frac{X^k-1}{X-1}.
\]
Passing to the completed polynomial via $X=1+T$ and denoting
$h_k(T):=g_k(1+T)$ we obtain the completed Hosokawa polynomial
\begin{equation}\label{L p q z formula}
\mathbb{L}_{p,q}^z(T)=\frac{h_{\alpha p'q'}(T)^d}
{h_{\alpha p'}(T)\,h_{\alpha q'}(T)}.
\end{equation}
By construction, $\mathbb{L}_{p,q}^z(T)\in\mathbb{Z}_\ell\llbracket T\rrbracket$
and the full completed Alexander polynomial is
\[
\widehat{\Delta}_{p,q}^z(T)=T^{d-1}\cdot \mathbb{L}_{p,q}^z(T).
\]
\par First, we compute the $\mu$ and $\lambda$ invariants associated to $\widehat{\Delta}_{p,q}^z(T)$ which are denoted $\mu_z(p,q)$ and $\lambda_z(p,q)$.

\begin{theorem}\label{thm 5.4}
    With respect to notation above, 
    \[\mu_z(p,q)=0\text{ and }\lambda_z(p,q)=(d-2)\,\ell^{\,v_\ell(\alpha)}.\]
\end{theorem}

\begin{proof}Because each $h_k(T)$ is a monic polynomial with integer coefficients, one immediately sees that $h_k(T)$ has $\mu$--invariant $0$.
Thus, from \eqref{L p q z formula}, we find that $\mu_z(p,q)=0$ for all $p,q,z$ in this family.

\par Let $\lambda_\ell(k)$ be the $\lambda$-invariant of $h_k(T)$. Write $k=\ell^{v_\ell(k)} k'$ with $\ell\nmid k'$.  Then
\[
h_k(T)
= \frac{(1+T)^{\ell^{v_\ell(k)}k'}-1}{T}
= h_{\ell^{v_\ell(k)}}(T)\cdot u_k(T),
\]
where
\[
u_k(T):=\sum_{j=0}^{k'-1} (1+T)^{\ell^{v_\ell(k)}j}.
\]
Since $u_k(0)=k'$ is prime to $\ell$ we have $u_k(T)\in\mathbb{Z}_\ell\llbracket T\rrbracket^\times$,
so $\lambda_\ell(k)=\lambda_\ell(\ell^{v_\ell(k)})$.  But for a pure power
$\ell^t$, it is well known that all binomial coefficients $\binom{\ell^t}{j}$ with
$1\le j<\ell^t$ are divisible by $\ell$, hence
\[
\lambda_\ell(\ell^t)=\deg h_{\ell^t}(T)=\ell^t-1,
\]
and in particular,
\[
\lambda_\ell(k)=\ell^{\,v_\ell(k)}-1.
\]
Applying this to the factorization of $\mathbb{L}_{p,q}^z(T)$ given by \eqref{L p q z formula}, deduce that:
\[
\begin{split}\lambda_z(p,q)
=& (d-1) + d\,\lambda_\ell(\alpha p'q') 
- \lambda_\ell(\alpha p') - \lambda_\ell(\alpha q')\\
= & (d-1) + d\big(\ell^{v_\ell(\alpha)}-1\big) 
- \big(\ell^{v_\ell(\alpha)}-1\big) - \big(\ell^{v_\ell(\alpha)}-1\big)\\
= & (d-2)\,\ell^{\,v_\ell(\alpha)} ,
\end{split}\]
\noindent where $\alpha=\sum_i z_i$ as above.
\end{proof}

\bibliographystyle{alpha}
\bibliography{references}

\begin{thebibliography}{GAnS91}

\bibitem[AM86]{ArtinMazur}
M.~Artin and B.~Mazur.
\newblock {\em Etale homotopy}, volume 100 of {\em Lecture Notes in Mathematics}.
\newblock Springer-Verlag, Berlin, 1986.
\newblock Reprint of the 1969 original.

\bibitem[FW79]{ferrerowash}
Bruce Ferrero and Lawrence~C. Washington.
\newblock The {I}wasawa invariant {$\mu _{p}$} vanishes for abelian number fields.
\newblock {\em Ann. of Math. (2)}, 109(2):377--395, 1979.

\bibitem[GAnS91]{acunashort}
Francisco Gonz\'{a}lez-Acu\~{n}a and Hamish Short.
\newblock Cyclic branched coverings of knots and homology spheres.
\newblock {\em Rev. Mat. Univ. Complut. Madrid}, 4(1):97--120, 1991.

\bibitem[Gor72]{Gordon}
C.~McA. Gordon.
\newblock Knots whose branched cyclic coverings have periodic homology.
\newblock {\em Trans. Amer. Math. Soc.}, 168:357--370, 1972.

\bibitem[Kap96]{kapranov1996analogies}
M~Kapranov.
\newblock Analogies between number fields and 3-manifolds.
\newblock {\em unpublished note}, 1996.

\bibitem[Lic97]{Lickorish}
W.~B.~Raymond Lickorish.
\newblock {\em An introduction to knot theory}, volume 175 of {\em Graduate Texts in Mathematics}.
\newblock Springer-Verlag, New York, 1997.

\bibitem[Maz63]{mazur1963remarks}
Barry Mazur.
\newblock Remarks on the alexander polynomial.
\newblock {\em Unpublished notes}, 64, 1963.

\bibitem[Mil68]{Milnor}
John Milnor.
\newblock {\em Singular points of complex hypersurfaces}.
\newblock Annals of Mathematics Studies, No. 61. Princeton University Press, Princeton, NJ; University of Tokyo Press, Tokyo, 1968.

\bibitem[Mor11]{morishita2011knots}
Masanori Morishita.
\newblock {\em Knots and primes: an introduction to arithmetic topology}.
\newblock Springer Science \& {B}usiness Media, 2011.

\bibitem[Nii14]{Nibo}
Hirofumi Niibo.
\newblock Id\`elic class field theory for 3-manifolds.
\newblock {\em Kyushu J. Math.}, 68(2):421--436, 2014.

\bibitem[NU19]{NiboUeki}
Hirofumi Niibo and Jun Ueki.
\newblock Id\`elic class field theory for 3-manifolds and very admissible links.
\newblock {\em Trans. Amer. Math. Soc.}, 371(12):8467--8488, 2019.

\bibitem[Rez97]{Reznikov1}
Alexander Reznikov.
\newblock Three-manifolds class field theory (homology of coverings for a nonvirtually {$b_1$}-positive manifold).
\newblock {\em Selecta Math. (N.S.)}, 3(3):361--399, 1997.

\bibitem[Rez00]{Reznikov2}
Alexander Reznikov.
\newblock Embedded incompressible surfaces and homology of ramified coverings of three-manifolds.
\newblock {\em Selecta Math. (N.S.)}, 6(1):1--39, 2000.

\bibitem[SW02]{SW1}
Daniel~S. Silver and Susan~G. Williams.
\newblock Mahler measure, links and homology growth.
\newblock {\em Topology}, 41(5):979--991, 2002.

\bibitem[SW04]{SW2}
Daniel~S. Silver and Susan~G. Williams.
\newblock Mahler measure of {A}lexander polynomials.
\newblock {\em J. London Math. Soc. (2)}, 69(3):767--782, 2004.

\bibitem[Web79]{Weber}
Claude Weber.
\newblock Sur une formule de {R}. {H}. {F}ox concernant l'homologie des rev\^{e}tements cycliques.
\newblock {\em Enseign. Math. (2)}, 25(3-4):261--272 (1980), 1979.

\end{thebibliography}
\end{document}